\documentclass{amsart}
\usepackage{amssymb}\usepackage{amsmath}\usepackage{a4}

\begin{document}
\newtheorem{theorem}{Theorem}[section]
\newtheorem{lemma}[theorem]{Lemma}
\newtheorem{remark}[theorem]{Remark}
\def\qedbox{\hbox{$\rlap{$\sqcap$}\sqcup$}}
\makeatletter
  \renewcommand{\theequation}{%
   \thesection.\alph{equation}}
  \@addtoreset{equation}{section}
 \makeatother
\def\Op{\operatorname{Op}}
\title[Singular heat content and heat trace asymptotics]
{Heat Content Asymptotics with singular initial temperature distributions}
\author{M. van den Berg, P. Gilkey, and R. Seeley}
\begin{address}{MvdB: Department of Mathematics, University of
Bristol, University Walk, Bristol,\newline\phantom{...a}BS8 1TW, U.K.}\end{address}
\begin{email}{M.vandenBerg@bris.ac.uk}\end{email}
\begin{address}{PG: Mathematics Department, University of Oregon, Eugene, OR 97403, USA}\end{address}
\begin{email}{gilkey@uoregon.edu}\end{email}
\begin{address}{RS: 35 Lakewood Rd, Newton, MA 02461, USA}\end{address}
\begin{email}{r-seeley@comcast.net}\end{email}
\begin{abstract} {We  study the heat content asymptotics with either Dirichlet or Robin boundary conditions where the initial
temperature exhibits radial blowup near the boundary. We show that there is a complete small-time asymptotic expansion and give explicit
geometrical formulas for the first few terms in the expansion.}
\end{abstract}
\maketitle
\section{Introduction}
\def\vandenberg{Setting
$S=0$ yields the Neumann boundary {operator}; {Neumann boundary conditions} correspond to a perfectly insulated boundary with
no heat loss.   It is perhaps worth pointing out in this setting that if either $\rho$ or $\phi$ are constant, then the heat content is
constant i.e. time independent. If $S = 0$ and if $\rho$ and $\phi$ are not constant, then the heat content is in general not constant.
This is  due to the fact that the initial temperature profile $\phi$ is ``flattening out" and that this ``flattening" out is integrated
in a non-uniform way in the integral for the heat content. Taking $S\ne0$
corresponds to a heat transfer across the boundary which is proportional to the temperature difference across the boundary. }

\subsection{The heat content}\label{sect-1.1}
Let $M$ be a compact Riemannian manifold of dimension
$m$ with smooth boundary
$\partial M$, let $dx$ and $dy$ be the Riemannian measures on $M$ and on $\partial M$, respectively, and
let $D$ be an operator of Laplace type on a smooth vector bundle $V$ over $M$. To impose suitable boundary conditions, we define
the {\it Dirichlet boundary operator} $B_{\mathcal{D}}\phi:=\phi|_{\partial M}$.
The operator $D$ defines a natural connection $\nabla$ as we shall discuss presently in Lemma \ref{lem-1.1}.
Let $S$ be an auxiliary endomorphism of the vector bundle $V|_{\partial M}$. Let
$${B_{\mathcal{R}}}\phi:=(\phi_{;m}+S\phi)|_{\partial M}$$
be the {\it Robin boundary operator} where $\phi_{;m}$ denotes the covariant derivative of $\phi$ with respect to the inward unit normal
vector field.
Let $B$ be either the Dirichlet or the Robin boundary operator; the associated boundary conditions are defined by setting $B\phi=0$.
It is well known that the heat equation
$$(\partial_t+D)u_B(x;t)=0,\quad
  Bu(\cdot;t)=0,\quad
  \lim_{t\downarrow 0}u_B(\cdot;t)=\phi(\cdot),
$$
 has a unique classical solution for a wide class of initial
temperature distributions $\phi$. We set $u=e^{-tD_B}\phi$ where
$D_B$ is the associated realization of $D$. The operator
$e^{-tD_B}$ has a kernel $p_B(x,\tilde x;t)$ which is smooth in
$(x,\tilde x;t)$ such that
$$u_B(x;t)=\int_Mp_B(x,\tilde x;t)\phi(\tilde x)d\tilde x\,.$$
If, for example, $D=\Delta:=\delta d$ is the scalar Laplacian, then one may take a complete spectral
resolution $\{\lambda_i,\phi_i\}$ of $\Delta_B$ and express
$$
p_B(x,\tilde x;t)=\sum_ie^{-t\lambda_i}\phi_i(x){{\bar\phi}}_i(\tilde x)\,.
$$

Let $\langle\cdot,\cdot\rangle$ denote
the natural pairing between $V$ and the dual bundle $\tilde V$. The {\it specific heat} $\rho$ of the
manifold  is a smooth section of $\tilde V$ and the {\it heat content}
$\beta$ is given by:
$$
\beta(\phi,\rho,D,B)(t):=\int_M\langle u_B(x;t),\rho(x)\rangle dx
   =\int_M\int_M\langle p_B(x,\tilde x;t)\phi(\tilde x),\rho(x)\rangle d\tilde xdx\,.
$$
Although in most practical applications it is customary to take $V$ and $\tilde V$ to be the trivial line
bundle and $D=\Delta$, it is necessary to work in this greater
generality as we shall see presently in Section \ref{sect-3}. {If $\partial M$ is empty, we shall omit the boundary
condition $B$ from the notation as it plays no role}. To simplify the notation, we shall write $D_B$, $p_B$, $\beta(\cdot,\cdot,D,B)$, and $u_B$ for the
most part except where it is useful to emphasize which boundary condition appears.

\subsection{Geometric preliminaries}\label{sect-1.2}
The following formalism will enable us to work in a tensorial and coordinate free fashion. We
adopt the {\it Einstein convention} and sum over repeated indices. Choose a system of local coordinates
$x=(x_1,...,x_m)$ for $M$ and choose a local trivialization of
$V$. Let $g_{\mu\nu}:=g(\partial_{x_\mu},\partial_{x_\nu})$ and let $g^{\mu\nu}$ be the inverse
matrix. As
$D$ is of Laplace type, there are matrices $A_1^\nu$ and $A_0$ so that:
\begin{equation}\label{eqn-1.a}
D=-\left\{g^{\mu\nu}\operatorname{Id}\partial_{x_\mu}\partial_{x_\nu}
+A_1^\nu\partial_{x_\nu}+A_0\right\}\,.
\end{equation}

If $\nabla$
is a connection on
$V$, we use $\nabla$ and the Levi--Civita connection to covariantly differentiate tensors of all
types and let `;' denote multiple covariant differentiation. We let $\phi_{;\mu\nu}$ be the components of
$\nabla^2\phi$. If $E$ is an auxiliary endomorphism of $V$, we define the associated {\it modified
Bochner Laplacian}:
$$D(g,\nabla,E)\phi:=-g^{\mu\nu}\phi_{;\nu\mu}-E\phi\,.$$
Let $\Gamma_{\mu\nu\sigma}$ and $\Gamma_{\mu\nu}{}^\sigma$ be the Christoffel symbols. We then have
\cite{G04}:
\begin{lemma}\label{lem-1.1}
If $D$ is an operator of Laplace type, then there exists a unique connection
$\nabla$ on $V$ and a unique endomorphism $E$ on $V$ so that $D=D(g,\nabla,E)$. The connection $1$-form
$\omega$ of $\nabla$ and the endomorphism $E$ are given by:
\begin{enumerate}
\item $\omega_\mu=\textstyle\frac12(g_{\mu\nu}A_1^\nu+g^{\sigma\varepsilon}\Gamma_{\sigma\varepsilon\mu}
\operatorname{Id})$.
\item $E=A_0-g^{\mu\nu}(\partial_{x_\nu}\omega_{\mu}+\omega_{\mu}\omega_{\nu}
    -\omega_{\sigma}\Gamma_{\mu\nu}{}^\sigma)$.
\end{enumerate}\end{lemma}

 We use the dual connection to covariantly differentiate the specific heat $\rho$; note that the connection
$1$ form
$\tilde\omega_\nu$ for $\tilde\nabla$ is the dual of $-\omega_\nu$. Thus
\begin{equation}\label{eqn-1.b}
  \tilde\nabla_{\partial_{x_\mu}}=\partial_{x_\mu}-
\textstyle\frac12(g_{\mu\nu}\tilde A_1^\nu+g^{\sigma\varepsilon}\Gamma_{\sigma\varepsilon\mu}\operatorname{id})
\quad\text{and}\quad
\tilde D\rho=-(g^{\mu\nu}\rho_{;\mu\nu}+\tilde E\rho)\,.
\end{equation}

 Near the
boundary, choose an orthonormal frame
$\{e_1,...,e_m\}$ for the tangent bundle so that $e_m$ is the inward unit geodesic normal; let indices
$a,b$ range from
$1$ to $m-1$ and index the induced orthonormal frame $\{e_1,...,e_{m-1}\}$ for the tangent
bundle of the boundary. We let `$:$' denote the components of tangential covariant differentiation defined by
$\nabla$ and the Levi-Civita connection of the boundary. Let $L_{ab}:=g(\nabla_{e_a}e_b,e_m)=\Gamma_{abm}$ be
the components of the second fundamental form. The difference between `$;$' and `$:$' is then
measured by $L$. For example, the following relation will prove useful subsequently:
\begin{equation}\label{eqn-1.c}
\begin{array}{l}
D\phi=-(\phi_{:aa}+\phi_{;mm}-L_{aa}\phi_{;m}+E\phi),\\
\tilde D\rho=-(\rho_{:aa}+\rho_{;mm}-L_{aa}\rho_{;m}+\tilde E\rho)\,.
\end{array}\end{equation}
Let $\operatorname{Ric}$ denote the Ricci tensor. Let ${\tilde B}_{\mathcal{R}}$ be the dual Robin boundary operator; it is defined by
the dual connection $\tilde\nabla$ and dual endomorphism $\tilde S$.

\subsection{Heat content asymptotics in the smooth setting}\label{sect-1.3}
One can use the calculus of pseudo-differential operators developed in
\cite{Gr68,Se69,Se69a} to show that:
\begin{theorem}\label{thm-1.2}
Let $\phi\in C^\infty(V)$ and let $\rho\in C^\infty(\tilde V)$.
There is a complete asymptotic expansion as
$t\downarrow0$ of the form:\par\noindent
$\displaystyle\beta(\phi,\rho,D,B)(t)\sim\sum_{n=0}^\infty\frac{(-t)^n}{n!}\int_M\langle\phi,\tilde
D^n\rho\rangle dx+
\sum_{k=0}^\infty t^{(1+k)/2}\int_{\partial M}\beta_{k}^{\partial M}(\phi,\rho,D,B)dy$.
\end{theorem}

The {\it heat content coefficients} $\beta_k^{\partial M}$ are
locally computable and are given by local geometric invariants.

\begin{theorem}\label{thm-1.3}
Let $\phi\in C^\infty(V)$ and let $\rho\in C^\infty(\tilde V)$.
\begin{enumerate}
\item With Dirichlet boundary conditions, one has that:
\begin{enumerate}\item
$\int_{\partial M}\beta_0^{\partial M}(\phi,\rho,D,{B_{\mathcal{D}}})dy=
     -\frac2{\sqrt\pi}\int_{\partial M}\langle\phi,\rho\rangle dy$.
\item $\int_{\partial M}\beta_1^{\partial M}(\phi,\rho,D,{B_{\mathcal{D}}})dy=
     -\int_{\partial M}\{\langle \phi_{;m},\rho\rangle-\frac12\langle L_{aa}\phi,\rho\rangle\}dy$.
\item $\int_{\partial M}\beta_2^{\partial M}(\phi,\rho,D,{B_{\mathcal{D}}})dy=
      -\frac2{\sqrt\pi}\int_{\partial M}\{\langle\frac23\langle\phi_{;mm},\rho\rangle+\frac23
\langle\phi,\rho_{;mm}\rangle$\smallbreak$\qquad-\frac23L_{aa}\langle \phi,\rho\rangle_{;m}
+\langle\phi E,\rho\rangle-\langle\phi_{:a},\rho_{:a}\rangle+
\langle(\frac1{12}L_{aa}L_{bb}$\smallbreak$\qquad-\frac16L_{ab}L_{ab}-\frac16\operatorname{Ric}_{mm})\phi,\rho\rangle\}dy$.
\end{enumerate}
\item With Robin boundary conditions, one has that:
\begin{enumerate}
\item $\int_{\partial M}\beta_0^{\partial M}(\phi,\rho,D,{B_{\mathcal{R}}})dy=0$.
\item $\int_{\partial M}\beta_1^{\partial M}(\phi,\rho,D,{B_{\mathcal{R}}})dy
=\int_{\partial M}\langle\phi,{\tilde B_{\mathcal{R}}}\rho\rangle
dy$.
\item $\int_{\partial M}\beta_2^{\partial M}(\phi,\rho,D,{B_{\mathcal{R}}})dy=\frac{4}{3\sqrt\pi}\int_{\partial
M}\langle{B_{\mathcal{R}}}\phi,{\tilde B_{\mathcal{R}}}\rho
\rangle dy$.
\end{enumerate}\end{enumerate}
\end{theorem}

With Dirichlet boundary conditions,  $\beta_3^{\partial M}$ is known and partial information
concerning $\beta_4^{\partial M}$ is available. With Neumann boundary conditions, $\beta_3$, $\beta_4$, and $\beta_5$
are known. We refer to \cite{G04} for further details; our indexing convention here is slightly different from that employed in
\cite{G04}.

\subsection{Singular initial temperatures}\label{sect-1.4}
We refer to \cite{Be07} for earlier work in the singular setting
 as the results of that paper provide the motivation and the starting point for this present work.
We begin by using the geodesic flow defined by the unit inward
normal vector field to define a diffeomorphism for some $\varepsilon>0$ between the collar
${\mathcal{C}_\epsilon}:=\partial M\times[0,\varepsilon]$ and a neighborhood of the boundary in $M$ which identifies
$\partial M\times\{0\}$ with $\partial M$;  the curves
$r\rightarrow (y_0,r)$ for $r\in[0,\varepsilon]$ are then unit speed geodesics perpendicular to the
boundary and $r$ is the geodesic distance to the boundary. {We fix a smooth cutoff function $\chi=\chi(r)$ on $\mathcal{C}_\epsilon$
so that $\chi=1$ near $r=0$ and so that $\chi=0$ near $r=\epsilon$.}

 Let $\nabla$ be a connection on a bundle $W$
over
${\mathcal{C}_\epsilon}$. Let
$\psi\in C^\infty(W|_{\partial M})$. We use parallel translation along the normal geodesic rays to extend
$\psi$ to a section of $W$ over ${\mathcal{C}_\epsilon}$. We shall denote this extension by $\psi(y)$ to emphasize the fact that
$\nabla_{\partial_r}\psi=0$ on ${\mathcal{C}_\epsilon}$. If $W=V$, we use the connection $\nabla$ defined by $D$;
if $W=\tilde V$, we use the dual connection $\tilde\nabla$ defined by $\tilde D$. We refer to \cite{G04} for further details.

Fix $\alpha\in\mathbb{C}$. Let
$\phi$ be a smooth section to
$V$ on the interior of
$M$ {such that}
$\phi r^\alpha\in C^\infty({\mathcal{C}_\epsilon})$; the parameter $\alpha$ controls the growth (if
$\operatorname{Re}(\alpha)>0$) or decay (if
$\operatorname{Re}(\alpha)<0$) of $\phi$ near the boundary {assuming that $\phi r^\alpha$ does not vanish identically on the
boundary.} We assume the specific heat is smooth so $\rho\in C^\infty(\tilde V)$. We may then expand $\phi$ and $\rho$ on ${\mathcal{C}_\epsilon}$
in the form:
\begin{equation}\label{eqn-1.d}
\phi(y,r)\sim\sum_{i=0}^\infty\phi_i(y)r^{i-\alpha}\quad\text{and}\quad
\rho(y,r)\sim\sum_{i=0}^\infty\rho_i(y)r^i\quad\text{as}\quad r\downarrow0\,.
\end{equation}
The coefficients $\phi_i$ and $\rho_i$ are then uniquely specified by the requirement that
$\nabla_{\partial_r}\phi_i=0$ and
$\tilde\nabla_{\partial_r}\rho_i=0$. If
$D=\Delta$, then the associated connection is flat and, if $\alpha=0$, the expansion of Equation
(\ref{eqn-1.d}) is just the usual Taylor series expansion of the functions $\phi$ and $\rho$. In particular,
$$
\begin{array}{lll}\phi_0=(r^\alpha\phi)|_{\partial M},&
  \phi_1=\{\nabla_{\partial r}(r^\alpha\phi)\}|_{\partial M},&
  \phi_2=\frac12\{(\nabla_{\partial r})^2(r^\alpha\phi)\}|_{\partial M}\\
\rho_0=\rho|_{\partial M},&
  \rho_1=\{\tilde\nabla_{\partial r}\rho\}|_{\partial M},&
  \rho_2=\frac12\{(\tilde\nabla_{\partial r})^2\rho\}|_{\partial M}\,.\vphantom{\vrule height 11pt}
\end{array}$$

Fix $t>0$. Let $x\in M$ and let $\tilde x=(\tilde
y,r)\in{\mathcal{C}_\epsilon}$. Suppose first that $B=B_{\mathcal
D}$ defines Dirichlet boundary conditions. Then $p_B(x,(\tilde
y,\tilde r),t)|_{\tilde r=0}=0$. Since $p_B$ is smooth for $t>0$
and $\mathcal{C}_\epsilon$ is compact, we may use the Taylor
series expansion of $p_B$ to derive the estimate:
$$|p_B(x,(\tilde r,\tilde y);t)|\le C(t)\tilde r\quad\text{on}\quad\mathcal{C}_\epsilon\,.$$
If $\operatorname{Re}(\alpha)<2$, then the integral
$u_B(x;t)=\int_Mp_B(x,\tilde x;t)\phi(\tilde x)d\tilde x$
is convergent and {bounded in $x$}. Consequently the heat content
$$\beta(\phi,\rho,D,B)(t):=\int_M\langle u_B(x;t),\rho(x)\rangle dx$$
is well defined for $t>0$.
If $\operatorname{Re}(\alpha)<1$, then $\phi\in L^1$ and the initial heat content
$\beta(\phi,\rho,D,B)(0)=\int_M\langle\phi,\rho\rangle dx$ is finite. If, however,
$1\le \operatorname{Re}(\alpha)$, then this integral may be divergent and the initial heat content can be
infinite. Still, sufficient cooling near the boundary takes place for $u_B$ to be in $L^1$ for any $t>0$.
A similar phenomenon occurs in the setting of non-compact Riemannian manifolds with infinite volume and with regular boundary and initial
temperature $\phi=1$ \cite{Be06,BeGi04}.

As this cooling phenomenon does not occur with Robin boundary conditions $B=B_{\mathcal R}$, we shall always assume
$\operatorname{Re}(\alpha)<1$ in this instance.

It is important to observe that although we are primarily interested in positive real
$\alpha$, it is necessary to consider complex values of $\alpha$ to justify some analytic continuation arguments. It is also necessary to
permit
$\operatorname{Re}(\alpha)<0$ to justify some computations in Sections \ref{sect-3} and \ref{sect-4}; these values are of interest in their
own right since
$\phi$ is not smooth if $\alpha$ is not an integer.

If $1\le\operatorname{Re}(\alpha)<2$, we must regularize the integral
$\int_M\langle\phi,\rho\rangle dx$ since the integral may be
divergent. The Riemannian measure is not in general product near the boundary. Since, however, $dx=dydr$ on the boundary of
$M$, we may decompose
$$\langle\phi,\rho\rangle dx=\langle\phi_0,\rho_0\rangle r^{-\alpha}dydr+O(r^{1-\alpha})\,.$$
For $\operatorname{Re}(\alpha)<2$, define:
\begin{eqnarray}\label{eqn-1.e}
&&{\mathcal{I}_{\operatorname{Reg}}}(\phi,\rho)
:=\int_{M-{\mathcal{C}_\epsilon}}\langle\phi,\rho\rangle dx
+\int_{\mathcal{C}_\epsilon}
\left\{\langle\phi,\rho\rangle dx-\langle\phi_0,\rho_0\rangle r^{-\alpha}dydr\right\}\\
&+&\int_{\partial M}\langle\phi_0,\rho_0\rangle dy
\times\left\{\begin{array}{lll}\frac{\varepsilon^{1-\alpha}}{1-\alpha}&\text{if}&\alpha\ne1,\\
\ln(\varepsilon)&\text{if}&\alpha=1\,.
\end{array}\right.\nonumber
\end{eqnarray}
This is clearly independent of $\varepsilon$ and agrees with $\int_M\langle\phi,\rho\rangle$ if $\operatorname{Re}(\alpha)<1$.
Briefly, the regularization ${\mathcal{I}_{\operatorname{Reg}}}(\phi,\rho)$ is a meromorphic
function of alpha with a simple pole at $\alpha = 1$. When $\alpha = 1$, then
${\mathcal{I}_{\operatorname{Reg}}}(\phi,\rho)$ is defined as the constant term in the Laurent expansion at
$\alpha = 1$, thus dropping the pole.

 The following is the main analytic result of this paper:

\begin{theorem}\label{thm-1.4}
If $B=B_{\mathcal{D}}$, assume $Re(\alpha)<2$; if $B=B_{\mathcal{R}}$, assume $Re(\alpha)<1$.
\ \begin{enumerate}\item
If $\alpha\ne1$, then there exists a full
asymptotic expansion as
$t\downarrow0$ {of the form}:
\begin{eqnarray*}
\beta(\phi,\rho,D,B)(t)&\sim&\sum_{n=0}^\infty
\frac{(-t)^n}{n!}{\mathcal{I}_{\operatorname{Reg}}}(\phi,\tilde D^n\rho)\\&+&\sum_{k=0}^\infty
t^{(1+k-\alpha)/2}\int_{\partial M}\beta_{k,\alpha}^{\partial M}(\phi,\rho,D,B)dy\,.
\end{eqnarray*}
\item If $\alpha=1$ (and hence $B=B_{\mathcal{D}}$), then there exists a full
asymptotic expansion as
$t\downarrow0$  {of the form}:{\begin{eqnarray*}
&&\beta(\phi,\rho,D,B)(t)\sim\sum_{n=0}^\infty
\frac{(-t)^n}{n!}{\mathcal{I}_{\operatorname{Reg}}}(\phi,\tilde D^n\rho)\\&&\qquad+\sum_{k=0}^\infty
t^{k/2}\int_{\partial M}\left\{\beta_{k,1}^{\partial
M}(\phi,\rho,D,B)+\ln(t)\check\beta_k^{\partial M}(\phi,\rho,D,B)\right\}dy\,.
\end{eqnarray*}}
\item There are natural tangential {bilinear} differential operators $\beta_{k,\alpha,i,j}^{\partial M}$, which are holomorphic for
$\alpha\ne1$, so that
\begin{eqnarray*}
&&\beta_{k,\alpha}^{\partial M}(\phi,\rho,D,B)=\sum_{i+j\le k}
   \beta_{k,\alpha,i,j}^{\partial M}(\phi_i,\rho_j,D,B)\,.
\end{eqnarray*}
\item $\displaystyle\int_{\partial M}\check\beta_k^{\partial
M}(\phi,\rho,D,B)dy=\left\{\begin{array}{lll}
-\frac12\frac{(-1)^n}{n!}\int_{\partial M}\langle\phi_0,(\tilde
D^n\rho)_0\rangle dy&\text{if}&k=2n,\\
0&\text{if}&\text{$k$ is odd}\,.\end{array}\right.$
\end{enumerate}
\end{theorem}

We note that $\tilde B_{\mathcal{R}}\rho=\rho_1+\tilde S\rho_0$. Theorem \ref{thm-1.3} generalizes to this setting to become:

\begin{theorem}\label{thm-1.5}
Set $c_\alpha:=2^{1-\alpha}\Gamma\left(\frac{2-\alpha}2\right)\frac1{\sqrt\pi(\alpha-1)}$.
\begin{enumerate}\item If $\alpha\ne1$, if $\operatorname{Re}(\alpha)<2$, and if $B=B_{\mathcal{D}}$, then:
\begin{enumerate}
\item $\int_{\partial M}\beta_{0,\alpha}^{\partial
M}(\phi,\rho,D,{B_{\mathcal{D}}})dy=c_\alpha\int_{\partial M}\langle\phi_0,\rho_0\rangle dy$.
\item $\int_{\partial M}\beta_{1,\alpha}^{\partial
M}(\phi,\rho,D,{B_{\mathcal{D}}})dy=c_{\alpha-1}\int_{\partial M}
     \langle\phi_1-{\textstyle{\frac12}}L_{aa}\phi_0,\rho_0\rangle dy$.
\smallbreak\item $\int_{\partial M}\beta_{2,\alpha}^{\partial M}(\phi,\rho,D,{B_{\mathcal{D}}})dy
   =c_{\alpha-2}\int_{\partial M}\{\langle\phi_2,\rho_0\rangle-\frac12\langle L_{aa}\phi_1,\rho_0\rangle
$\smallbreak\ \ $
    -\frac{\alpha-3}{2(\alpha-1)(\alpha-2)}\langle E\phi_0,\rho_0\rangle
   +\frac2{(\alpha-1)(\alpha-2)}\langle\phi_0,\rho_2\rangle
   -\frac1{(\alpha-1)(\alpha-2)}\langle L_{aa}\phi_0,\rho_1\rangle
$\smallbreak\ \ $+\frac{\alpha-3}{2(\alpha-1)(\alpha-2)}\langle\phi_{0:a},\rho_{0:a}\rangle
    -\frac{\alpha-1}{4(\alpha-2)}\langle\operatorname{Ric}_{mm}\phi_0,\rho_0\rangle$\smallbreak\ \ $
+\frac{\alpha-1}{8(\alpha-2)}\langle L_{aa}L_{bb}\phi_0,\rho_0\rangle
-\frac{\alpha-1}{4(\alpha-2)}
\langle L_{ab}L_{ab}\phi_0,\rho_0\rangle\}dy$.
\end{enumerate}
\smallbreak\item
Let $\gamma$ be Euler's constant. If $\alpha=1$ and if $B=B_{\mathcal{D}}$, then
\medbreak
$\beta(\phi,\rho,D,B_{\mathcal{D}})(t)\sim
\mathcal{I}_{\operatorname{Reg}}(\phi,\rho)-{\textstyle\frac12}\ln(t)\int_{\partial M}\langle\phi_0,\rho_0\rangle dy
+\int_{\partial M}{\textstyle\frac\gamma2}\langle\phi_0,\rho_0\rangle dy$
\smallbreak\qquad$+t^{1/2}\int_{\partial M}\{
-{\textstyle\frac2{\sqrt\pi}}\langle\phi_1,\rho_0\rangle +{\textstyle\frac1{\sqrt\pi}}\langle
L_{aa}\phi_0,\rho_0\rangle\}dy$
\smallbreak\qquad$+t\{-\mathcal{I}_{\operatorname{Reg}}(\phi,\tilde D\rho)+{\textstyle\frac12}\ln(t)\int_{\partial M}
\langle\phi_0,(\tilde D\rho)_0\rangle dy\}$
\smallbreak\qquad$+t\int_{\partial M}\{\frac\gamma2\langle\phi_0,-(\tilde D\rho)_0\rangle
-\langle\phi_2,\rho_0\rangle
+{\textstyle\frac12}\langle L_{aa}\phi_1,\rho_0\rangle
+\langle\phi_0,\rho_2\rangle$\smallbreak\qquad\qquad$
-\frac12\langle L_{aa}\phi_0,\rho_1\rangle\}dy$
$+O(t^{3/2})$.
\item If $\alpha\ne0$, if $\operatorname(Re)(\alpha)<1$, and if $B=B_{\mathcal{R}}$, then:
\begin{enumerate}
\smallbreak\item $\int_{\partial M}\beta_{0,\alpha}^{\partial
M}(\phi,\rho,D,{B_{\mathcal{R}}})dy=0$. \smallbreak\item
$\int_{\partial M}\beta_{1,\alpha}^{\partial
M}(\phi,\rho,D,{B_{\mathcal{R}}})dy=\frac{2\alpha}{2-\alpha}c_{\alpha+1}\int_{\partial
M}\langle\phi_0,\tilde B_{\mathcal{R}}\rho\rangle dy$.
\smallbreak\item $\int_{\partial M}\beta_{2,\alpha}^{\partial
M}(\phi,\rho,D,{B_{\mathcal{R}}})dy
$\smallbreak\quad$=\frac{-2}{3-\alpha}c_\alpha\int_{\partial
M}\langle(1-\alpha) \phi_1+S\phi_0-\frac\alpha2L_{aa}\phi_0,\tilde
B_{\mathcal{R}}\rho\rangle dy$.
\end{enumerate}\end{enumerate}
\end{theorem}

\begin{remark}\label{rmk-1.6}
\rm We note that setting $\alpha=0$ in Assertion (1) of Theorem
\ref{thm-1.5} yields Assertion (1) of Theorem \ref{thm-1.3} and
that taking the limit as $\alpha\rightarrow0$ in Assertion (3) of
Theorem \ref{thm-1.5} yields Assertion (2) of Theorem
\ref{thm-1.3}.
\end{remark}
To avoid subscripts on subscripts, we shall for the most part simply refer to $u_B$, $D_B$, and $p_B$ when no danger of confusion
is likely to ensue; however, we shall use the notation $D_{B_\mathcal{D}}$ and $D_{B_\mathcal{R}}$ in Section $\ref{sect-4}$ when we must
deal with two different boundary conditions.

\subsection{Outline of the paper}\label{sect-1.5}
In Section
\ref{sect-2}, we use the calculus of
pseudo-differential operators to establish Theorem
\ref{thm-1.4} (1-3); we postpone the proof of (4) as it will follow as a scholium to the proof of Lemma \ref{lem-3.1};
it can also be deduced directly from Lemma \ref{lem-2.6}. We shall restrict to Dirichlet boundary conditions as the analysis is similar
for Robin boundary conditions. In Section
\ref{sect-3}, we apply invariance theory and the functorial method to prove Assertions (1) and (2) of Theorem \ref{thm-1.5}; an essential
input is the calculation of a single coefficient performed in \cite{Be07} and a special case computation on the half-line.  In Section
\ref{sect-4}, we establish Assertion (3) of Theorem \ref{thm-1.5}. We plan in a
subsequent paper \cite{BeGiKiSe08} to undertake a similar analysis of the heat trace asymptotics with a singular smearing function.

The special case $\alpha = 1$ in Theorem \ref{thm-1.5} can be
derived from the case $\alpha\ne1$ by a straightforward but
tedious analytic continuation; the poles at $\alpha =1$ which
arise from the terms in Assertion 1 of Theorem \ref{thm-1.5} are
cancelled by those involved in the regularization
${\mathcal{I}_{\operatorname{Reg}}}(\phi,\tilde D^n\rho)$. The
derivation of the $t^0$ and $t^0\ln(t)$ terms in this way is
sketched in Section \ref{sect-2.1}; the other terms use Equations
(\ref{eqn-1.b}), (\ref{eqn-1.c}), and (\ref{eqn-3.b}) below. We
will give an independent derivation of the case $\alpha = 1$ to
show how  scaling arguments can be applied when logarithmic terms
are present.

\section{Pseudo-Differential Operators}\label{sect-2}
We suppose throughout Section \ref{sect-2} that $\operatorname{Re}(\alpha)<2$ and that $B=B_{\mathcal{D}}$. We use the identity
$$\int_M\langle e^{-tD_B}\phi,\rho\rangle dx=\int_M\langle\phi,e^{-t\tilde D_B}\rho\rangle
dx$$ to interchange the roles of $\phi$ and $\rho$. {This causes no difficulty as
$$\beta(\phi,\rho,D,B_{\mathcal{D}})(t)=\beta(\rho,\phi,\tilde D,B_{\mathcal{D}})(t)\,.$$}
We shall assume
throughout this section that $\phi$ is smooth and that
$$
\rho(y,r)=\rho_0(y){\chi(r)}r^{-\alpha}
  \quad\text{on}\quad{\mathcal{C}_\epsilon}\,.
$$
The more general case where $\rho\sim(\rho_0(y)+\rho_1(y)r+...)\chi(r)r^{-\alpha}$ then follows similarly; if $\phi$ and $\rho$
vanish to high order on the boundary, the corresponding boundary contributions vanish to high order in
$t$. Frequently in this section,
we will let $\beta$ be a multi-index rather than denoting the total heat content; we shall also let $\gamma$ both be a multi-index and
Euler's constant. We apologize in advance for any confusion this may cause. We let ${\mathcal{I}_{\operatorname{Reg}}}$ be the
regularization defined in Equation (\ref{eqn-1.e}) where we interchange the roles of $\phi$ and $\rho$.

Our fundamental analytical result is the following:

\begin{theorem}\label{thm-2.1}
Adopt the notation established above. There are differential operators $\mathbb{G}_k$ of order at most $k$
defined on
$M$ and tangential differential operators
${\mathbb{B}_k(\alpha)}$ and
$\mathbb{L}_k$ of order at most $k$ defined on $\partial M$ such that in closed subsectors of
$\{\mu\in\mathbb{C}:|\operatorname{arg}(\mu)|<\frac12\pi\}$ one has:
\begin{enumerate}\item Let $\alpha\ne1$ and let $\operatorname{Re}(\alpha)<2$. Then
$\mathbb{B}_0(\alpha)=-\Gamma(1-\alpha)$, $\mathbb{B}_k(\alpha)$ is holomorphic in $\alpha$,  and
$$
\langle(D_B+\mu^2)^{-1}\phi,\rho\rangle\sim\sum_{k=0}^\infty\mu^{-2-k}
{\mathcal{I}_{\operatorname{Reg}}}(\mathbb{G}_k\phi,\rho)
+\sum_{k=0}^\infty\mu^{\alpha-3-k}\int_{\partial
M}\langle\mathbb{B}_k(\alpha)\phi,\rho_0\rangle dy \,.
$$
\item If $\alpha=1$, then:
\begin{eqnarray*}
\langle(D_B+\mu^2)^{-1}\phi,\rho\rangle
&\sim&\sum_{k=0}^\infty\mu^{-2-k}\left\{{\mathcal{I}_{\operatorname{Reg}}}(
\mathbb{G}_k\phi,\rho) +\int_{\partial
M}\langle\mathbb{B}_k(1)\phi,\rho_0\rangle
dy\right.\\&&\left.\qquad+ \ln(\mu)\int_{\partial
M}\langle\mathbb{L}_k\phi,\rho_0\rangle dy\right\} \,.
\end{eqnarray*}
\end{enumerate}\end{theorem}

Here is a brief outline of the proof. In Section \ref{sect-2.1},
we study the half-line $\mathbb{R}^1_+$. In Section
\ref{sect-2.2}, we study the half-space $\mathbb{R}^m_+$. In
Section \ref{sect-2.3}, we complete the proof by considering the
case of manifolds. Theorem \ref{thm-1.4} then follows. The
interior integrals are evaluated using previous techniques.
Furthermore, for $\alpha\ne1$, one has
$$\beta_{k,\alpha}^{\partial M}(\phi,\rho,D,B_{\mathcal{D}})=\Gamma
{\textstyle\left(\frac{3-\alpha+k}2\right)}^{-1}
{\langle\mathbb{B}_k(\alpha)\phi,\rho_0\rangle}\,.$$
In particular, we may use the duplication formula for the Gamma function to establish Assertion (1a) in Theorem
\ref{thm-1.5} by computing:
\begin{eqnarray*}
&&\textstyle\beta_{0,\alpha}^{\partial M}(\phi,\rho,D,B_{\mathcal{D}})
=-\frac{\Gamma(1-\alpha)}{\Gamma\left(\frac{3-\alpha}2\right)}\langle\phi_0,\rho_0\rangle
=\textstyle2^{1-\alpha}\Gamma\left(\frac{2-\alpha}2\right)
  \frac1{\sqrt\pi(\alpha-1)}\langle\phi_0,\rho_0\rangle\,.
\end{eqnarray*}

\subsection{The half line}\label{sect-2.1}
{The case of $\mathbb{R}^1_+$ gives the basic outline of the proof in the general case}. Let $\eta$ be the Fourier
transform variable related to $r$. The Fourier inversion formula then becomes:
$$\hat\phi(\eta)=\int_{0}^\infty e^{-\sqrt{-1}r\eta}\phi(r)dr\quad\text{and}\quad
\phi(r)=\int_{-\infty}^\infty e^{\sqrt{-1}r\eta}\hat\phi(\eta)\bar d\eta$$
where we set $\bar d\eta:=d\eta/2\pi$.
For
$D=-\partial_r^2$, we may then write the {Dirichlet resolvent} as a pseudo-differential operator in the form:
\begin{eqnarray*}
(D_B+\mu^2)^{-1}\phi(r)&=&\int_{-\infty}^\infty e^{\sqrt{-1}r\eta}c_{-2}(\eta,\mu)\hat\phi(\eta)\bar d\eta
+\int_{-\infty}^\infty d_{-2}(r,\eta,\mu)\hat\phi(\eta)\bar d\eta\nonumber\\
&=:&[\Op(c_{-2})+\Op^\prime(d_{-2})]\phi(r)\,.
\end{eqnarray*}
Here  $c_{-2}=(\eta^2+\mu^2)^{-1}$, and
$d_{-2}(r,\eta,\mu)=-(\eta^2+\mu^2)^{-1}e^{-\mu r}$ is the bounded solution of
$$
(-\partial_r^2+\mu^2)d_{-2}=0\quad\text{and}\quad
d_{-2}|_{r=0}=-c_{-2}(\eta,\mu)\,.
$$
{The kernel $k(r,s,\mu)$ of $(\Delta_B+\mu^2)^{-1}$ is thus, for
$r>0$ and $s>0$, given by}
$$k(r,s,\mu)=\frac1{2\mu}\left[e^{-|r-s|\mu}-e^{-(r+s)\mu}\right]
\,.$$

We now consider the heat conduction problem:
$$\phi(r)\equiv1\quad\text{and}\quad\rho(r)=\chi(r)r^{-\alpha},\quad\text{where}\quad\chi\quad\text{is
smooth and}\quad $$
$$ \chi(r)=\left\{\begin{array}{ll}
1,&0<r<1/3,\\
0, &2/3<r.\end{array}\right.$$

Suppose that
$\operatorname{Re}(\alpha)<1$. One then has that:
\begin{eqnarray}
&&\mu^2\langle(D_B+\mu^2)^{-1}\phi,\rho\rangle
 = \mu^2\int_0^\infty\chi(r)r^{-\alpha}\int_0^\infty k(r,s,\mu)dsdr\label{eqn-2.a}\\
&=&\int_0^\infty\chi(r)r^{-\alpha}(1-e^{-r\mu})dr\nonumber\\
&=&\int_0^1r^{-\alpha}dr + \int_0^1(\chi(r) - 1)r^{-\alpha}dr
-\int_0^\infty r^{-\alpha}e^{-r\mu}dr\nonumber
\\&-&\int_{\frac13}^\infty(\chi(r)-1)e^{-r\mu}r^{-\alpha}dr\nonumber \\
&=& \left[{\frac1{(1-\alpha)}} + \int_0^1(\chi(r) -
1)r^{-\alpha}dr\right] - \mu^{\alpha-1}\Gamma(1-\alpha) +
O(\mu^{-\infty})\,.\nonumber
\end{eqnarray}
The term in bracket gives
$\mathcal{I}_{\operatorname{Reg}}(\phi,\rho)$ for all $\alpha \neq
1$, hence
\begin{equation}\label{eqn-2.b}
\mu^2\langle(D_B + \mu^2)^{-1}\phi,\rho\rangle =
\mathcal{I}_{\operatorname{Reg}}(\phi,\rho) - \mu^{\alpha - 1}
\Gamma(1-\alpha) +O(\mu^{-\infty}), \alpha \neq 1\end{equation} As
$\alpha\rightarrow 1$ we note that $\Gamma(1-\alpha) =
\frac1{1-\alpha} - \gamma + ...$ where $\gamma$ is Euler's
constant. Consequently
$$1/(1-\alpha) - \mu^{\alpha-1}\Gamma(1-\alpha) \rightarrow \ln(\mu) + \gamma\,.$$ So as
$\alpha\rightarrow1$, Equation (\ref{eqn-2.a}) gives
\begin{equation}\label{eqn-2.c}
\mu^2\langle(D_B +\mu^2)^{-1}\phi,\rho\rangle = \ln(\mu) + \gamma
+ \mathcal{I}_{\operatorname{Reg}}(\phi,\rho) +
O(\mu^{-\infty}),\quad \alpha = 1.
\end{equation}

We pass to the heat content by a contour integral
$$\beta(\phi,\rho,D,B_{\mathcal{D}})(t)=
\frac1{2\pi\sqrt{-1}}\int_Ce^{{t\lambda}}\langle(D_B+\lambda)^{-1}\phi,\rho\rangle
d\lambda,$$ where $C$ is a ``keyhole contour'' consisting of two
rays $\{re^{\pm\sqrt{-1}(\pi-\epsilon)},r\ge R\}$ and a circular
arc $\{Re^{\sqrt{-1}\theta},|\theta|\le\pi-\epsilon\}$. We then
get
\begin{eqnarray}\label{eqn-2.d1}
\qquad\beta(\phi,\rho,-\partial_r^2,B_D)(t)=\left\{\begin{array}{ll}
\mathcal{I}_{\operatorname{Reg}}(\phi,\rho)  -\frac{
\Gamma(1-\alpha)}{\Gamma(\frac12(3-\alpha))} + O(t^\infty),
& (\alpha\ne 1)\\
\mathcal{I}_{\operatorname{Reg}}(\phi,\rho)
-{\textstyle\frac12}\ln(t) + {\textstyle\frac12}\gamma +
O(t^\infty),& (\alpha = 1)\end{array}\right.
\end{eqnarray}
This last formula is valid also for the same functions $\phi$ and
$\rho$ on the interval $[0,1]$ with Dirichlet conditions at both
ends; since $\rho\equiv0$ near $r = 1$, the boundary correction
from $r=1$ is $O(t^\infty)$. The constant appearing in
\eqref{eqn-2.d1} for $\alpha\neq 1$ jibes with the constant
$c_{\alpha}$ defined in Theorem \ref{thm-1.5} by Legendre's
duplication formula. Thus, for this special case, we have
established Assertions (1) and (2) of Theorem \ref{thm-1.5}.
\subsection{The half-space $\mathbb{R}^m_+$}\label{sect-2.2}
We use the resolvent construction described in \cite{Se69a}. Let $\xi=(\xi_1,...,\xi_{m-1})$ be the Fourier
transform variables dual to $y=(y_1,...,y_{m-1})$. Note that
$g(\partial_{y_i},\partial_r)=0$ and
$g(\partial_r,\partial_r)=1$. We set
$$q^2(x,\xi):=g(\xi_ady^a,\xi_bdy^b)=g^{ab}(x)\xi_a\xi_b\,.$$
We adopt the notation of Equation (\ref{eqn-1.a}). The symbol of $D$ is given by
$$\begin{array}{lll}
\sigma(D)=\sigma_2+\sigma_1+\sigma_0,&\text{where}&
\sigma_2=\left\{q^2(x,\xi)+\eta^2\right\}\operatorname{id},\\
\sigma_1=-\sqrt{-1}\{A_1^a\xi_a+A_1^m\eta\},&\text{and}&
\sigma_0=-A_0\,.\vphantom{\vrule height 12pt}
\end{array}$$

The resolvent parametrix is the sum of an interior part and a boundary correction. Let
$y\cdot\xi=y_1\xi_1+...+y_{m-1}\xi_{m-1}$. The interior part of the parametrix is a finite sum
$\Op(c_{-2})+...+\Op(c_{-N})$ where
\begin{equation}\label{eqn-2.d}
[\Op(c)\phi](y,r)=\int\int e^{\sqrt{-1}(y\cdot\xi+r\eta)}c(y,r,\xi,\eta)\hat\phi(\xi,\eta)\bar d\eta\bar
d\xi\,.
\end{equation}
The leading term in the interior of $M$ is
$$c_{-2}(y,r,\xi,\eta,\mu)=(\sigma_2+\mu^2)^{-1}$$
and successive terms $c_{-3},...$ are defined by the usual pseudo-differential calculus with parameter
$\mu$; see, for example, Lemma 1.7.2 \cite{G95}. In particular, $c_{j}$ is
homogeneous of degree $j$ in the variables $(\xi,\eta,\mu)$ and, for
$|\operatorname{arg}(\mu)|\le\frac12\pi-\epsilon$, we have the estimate:
\begin{equation}\label{eqn-2.e}
|\partial_{(y,r)}^\beta\partial_{(\xi,\eta)}^\gamma c_{j}|\le
\text{const}_{\beta,\gamma,\epsilon,j}|(|\xi|+|\eta|+|\mu|)^{j-|\gamma|}
\,.
\end{equation}

The boundary part of the parametrix is a finite sum of operators
$$\Op^\prime(d_j)={\int_{\mathbb{R}}}\int_{\mathbb{R}^{m-1}}e^{\sqrt{-1}y\cdot\xi}
d_j(y,r,\xi,\eta,\mu)\hat\phi(\xi,\eta)\bar d\xi\bar d\eta
$$
which are chosen so that
\begin{enumerate}
\item $D\{\Op^\prime(d_{-2})+...+\Op^\prime(d_{-N})\}$ has order $1-N$,
\item $[\Op(c_j)+\Op^\prime(d_j)]\phi(y,0)=0$.
\end{enumerate}
To ensure that this second condition is satisfied, we set:
\begin{equation}\label{eqn-2.f}
d_j(y,0,\xi,\eta,\mu)=-c_j(y,0,\xi,\eta,\mu)\,.
\end{equation}
To ensure the first condition is satisfied, we begin by setting
$$[-\partial_r^2+q^2+\mu^2]d_{-2}=0\,.$$
Then equation (\ref{eqn-2.f}) yields
$$
d_{-2}(y,r,\xi,\eta,\mu)=-(q^2+\eta^2+\mu^1)^{-1}e^{-r\sqrt{q^2+\mu^2}}\,.
$$
We may then define $d_j$ inductively by an equation of the form:
$${(-\partial_r^2+q^2+\mu^2)d_{j-1}=\sum_{-k=2}^{-j}a_{jklu\beta\gamma}(y)r^\ell\partial_r^u
\xi^\beta\partial_y^\gamma d_k},
$$
where $k-\ell+{u}+|\beta|=j+1$. The coefficients $a_{jklu\beta\gamma}(y)$ come from the Taylor {expansion} of
the coefficients of $D$ in powers of $r$, see \cite{Se69a} for details. For some constant $C$, we have:
\begin{eqnarray}
&&d_j(y,r/t,t\xi,t\eta,t\mu)=t^jd_j(y,r,\xi,\eta,\mu),\label{eqn-2.g}\\
&&r^k\partial_\eta^u\partial_\xi^\beta\partial_r^\ell\partial_y^\gamma d_j
=O(|\xi|+|\eta|+|\mu|)^{-2-u}(|\xi|+|\mu|)^{j-k-|\beta|+\ell+2}
e^{-Cr(|\xi|+|\mu|)}\,.\label{eqn-2.h}
\end{eqnarray}

Now consider the expansion of $[\Op(c_j)\phi](y,r)$ for $\phi\in\mathcal{S}(\mathbb{R}^m_+)$, i.e. assume
that
$\phi$ has an extension in the Schwarz class $\mathcal{S}(\mathbb{R}^m)$. Set

$$\phi^{(\beta)}:=\partial_{(y,r)}^\beta\phi(y,r)$$ and
$$c_{j\beta}(y,r,\xi,\eta,\mu):=(-\sqrt{-1})^{|\beta|}\partial_{(\xi,\eta)}^\beta
c_j(y,r,\xi,\eta,\mu)\,.
$$
\begin{lemma}\label{lem-2.2}
As $\mu\rightarrow\infty$ in closed subsectors of $\{\mu\in\mathbb{C}:|\operatorname{arg}(\mu)|<{\frac12\pi}\}$,
\begin{eqnarray*}
&&[\Op(c_j)\phi](y,r)\sim\sum_\beta\frac{1}{\beta!}[C_{j\beta}+A_{j\beta}]\phi^{(\beta)}(y,r),
\quad\text{with}\\
&&C_{j\beta}(y,r,\mu)=c_{j\beta}(y,r,0,0,\mu),\quad\text{and}\\&&
A_{j\beta}(y,r,\mu)=-\int_{-\infty}^{-r}\int_{-\infty}^\infty e^{-\sqrt{-1}s\zeta}
c_{j\beta}(y,r,0,\zeta,\mu)\bar d\zeta ds\,.
\end{eqnarray*}
\end{lemma}

\begin{proof} In Equation (\ref{eqn-2.d}), consider first the integral $\bar d\xi$. Let $\Phi(y,\eta)$
be the Fourier transform in $r$. {A Taylor expansion of $\Phi(y,\eta)$ in powers of $y-\tilde y$}
gives
\begin{eqnarray}
&&\int_{-\infty}^\infty e^{\sqrt{-1}y\cdot\xi}c_j(y,r,\xi,\zeta,\mu)
\int_{-\infty}^\infty e^{-\sqrt{-1}\tilde y\cdot\xi}\Phi({\tilde y},\zeta)d{\tilde y}\bar d\xi\nonumber\\
&=&\sum_{|\gamma|<K}\frac1{\gamma!}c_{j\gamma}(y,r,0,\zeta,\mu)\Phi^{(\gamma)}(y,\zeta)\nonumber\\
&+&\int\int e^{\sqrt{-1}(y-\tilde y)\cdot\xi}\sum_{|\gamma|=K}c_{j\gamma}(y,r,\xi,\zeta,\mu)
R_{\gamma}(y,r,\tilde y,\zeta)\bar d\xi d{\tilde  y}\,.\label{eqn-2.i}
\end{eqnarray}
In this expression, $\gamma=(\gamma_1,...,\gamma_{m-1},0)$ is a multi-index. The change in the order of integration is
clearly justified if $j<1-m$. For other $j$, we may insert a factor of {$(1+|\xi|^2)^w$ and continue
analytically to $w=0$ from $2\operatorname{Re}(w)+j<1-m$.}

We now multiply Equation (\ref{eqn-2.i}) by $e^{\sqrt{-1}r\zeta}$ and integrate $\bar d\zeta$. The remainder
integral is a harmless $O(|\mu|^{j-k+m+2})$ since $R_j$ is bounded and for all $\gamma^\prime$ and for
$|\gamma|=K$, we have
$$(y-\tilde y)^{\gamma^\prime}\int e^{\sqrt{-1}(y-\tilde y)\cdot\xi}c_{j\gamma}
(y,r,\xi,\zeta,\mu)\bar d\xi=O(|\zeta|+\mu|)^{j-k+m}\,.$$
From the terms with $|\gamma|<K$, we have
\begin{eqnarray*}
&&\frac1{\gamma!}\int_{-\infty}^\infty e^{\sqrt{-1}r\zeta}c_{j\gamma}(y,r,0,\zeta,\mu)\Phi^{(\gamma)}
(y,{\zeta}){\bar d\zeta}\\
&=&\frac1{\gamma!}\int_0^\infty\int_{-\infty}^\infty e^{\sqrt{-1}(r-s)\zeta}c_{j\gamma}(y,r,0,\zeta,\mu)
\phi^{(\gamma)}({y,s})\bar d\zeta ds\,.
\end{eqnarray*}

A Taylor expansion of $\phi^{(\gamma)}$ in powers of $s-r$ gives
$$\sum_u\frac1{u!\gamma!}\int_0^\infty\int_{-\infty}^\infty
e^{\sqrt{-1}(r-s)\zeta}c_{j\beta}(y,r,0,\zeta,\mu)\bar d\zeta ds\phi^{(\beta)}(y,r)$$
plus a harmless remainder. Here $\beta=(\gamma_1,...,\gamma_{m-1},u)$. Writing
$\int_0^\infty$ as $\int_{-\infty}^\infty-\int_{-\infty}^0$, we find the $C_{j\beta}$ and $A_{j\beta}$ as
in Lemma \ref{lem-2.2}.\end{proof}

We continue our development. Let $d_{j\beta}=(-\sqrt{-1})^{|\beta|}\partial_{\xi,\zeta}^\beta d_j$.

\begin{lemma}\label{lem-2.3}
$[\Op^\prime(d_j)\phi](y,r)\sim\sum_\beta\frac1{\beta!}B_{j\beta}\phi^{(\beta)}(y,r,\mu)$ with
$$B_{j\beta}(y,r,\mu)=\int_0^\infty\int_{-\infty}^\infty
e^{-\sqrt{-1}s\zeta}d_{j\beta}(y,r,0,\zeta,\mu){\bar d\zeta} ds\,.$$
Furthermore
\begin{equation}\label{eqn-2.j}
[A_{j\beta}+B_{j\beta}+C_{j\beta}](y,0,\mu)=0\,.
\end{equation}\end{lemma}
\begin{proof} The proof is similar to the proof given in Lemma \ref{lem-2.2}. We sketch the details
as follows. We express
\begin{eqnarray*}
&&[\Op^\prime(d_j)\phi](y,r)=\int\int\int\int e^{\sqrt{-1}(y-\tilde y)\cdot\xi-\sqrt{-1}s\zeta}
d_j(y,r,\xi,\zeta,\mu)\phi(\tilde y,s){d\tilde y}ds\bar d\zeta \bar d\xi
\end{eqnarray*}
and expand $\phi(\tilde y,s)$ in powers of $(\tilde y-y,s-r)$. To establish Equation (\ref{eqn-2.j}), we use
Equation (\ref{eqn-2.f}) to see that:
\begin{eqnarray*}
[A_{j\beta}+B_{j\beta}](y,0,\mu)&=&-\int_{-\infty}^0\int_{-\infty}^\infty e^{-\sqrt{-1}s\zeta}
{c}_{j\beta}(y,0,0,\zeta,\mu)\bar d\zeta ds\\
&-&\int_0^\infty\int_{-\infty}^\infty e^{-\sqrt{-1}s\zeta}c_{j\beta}(y,0,0,\zeta,\mu)\bar d\zeta ds\\
&=&-c_{j\beta}(y,0,0,0,\mu)=-C_{j\beta}(y,0,\mu).
\end{eqnarray*}
The Lemma now follows.
\end{proof}

Next, we consider $\langle (D_B+\mu^2)^{-1}\phi,\rho\rangle$ as $\mu\rightarrow\infty$ where
$\rho(y,r)=\rho_0(y)r^{-\alpha}$. Define
\begin{equation}\label{eqn-2.k}
A_{j\beta}^\sim(y,r,t,\mu):=-\int_{-\infty}^{-t}\int_{-\infty}^\infty e^{-\sqrt{-1}s\zeta}
c_{j\beta}(y,r,0,\zeta,\mu)\bar d\zeta ds\,.
\end{equation}

We set $\mu^\prime:=\mu/|\mu|$. We then have
\begin{equation}\label{eqn-2.L}
A_{j\beta}(y,r,\mu)=|\mu|^{j-|\beta|}A_{j\beta}^\sim(y,r,|\mu| r,\mu^\prime)\,.
\end{equation}

\begin{lemma}\label{lem-2.4}
We have:
\begin{eqnarray*}
&&\int_0^\infty\langle A_{j\beta}(y,r,\mu)\phi^{(\beta)}(y,r),\rho_0(y)\rangle r^{-\alpha}dr\\
&\sim&\sum_{k=0}^\infty|\mu|^{{\alpha-1+j-k-|\beta|}}\int_0^\infty
\frac1{k!}\langle\partial_r^k[A_{j\beta}^\sim(y,r,t,\mu^\prime)\phi^{(\beta)}(y,r)],\rho_0(y)\rangle|_{r=0}
t^{k-\alpha}dt\,.
\end{eqnarray*}
The remainder after $N$ terms of the {expansion} is analytic for
$\operatorname{Re}(\alpha)<N$ and is $O(|\mu|^{1+j-N-|\beta|})$
uniformly in $|\operatorname{Re}(\alpha)|<N$.
\end{lemma}
\begin{proof} With a change of variable $|\mu|r=t$, we have
\begin{eqnarray*}
&&\int_0^\infty\langle A_{j\beta}\phi^{(\beta)},{\rho_0}\rangle r^{-\alpha}dr\\
&=&|\mu|^{\alpha-1+j-|\beta|}\int_0^\infty\langle
A_{j\beta}^\sim(y,t/|\mu|,t,\mu^\prime)\phi^{(\beta)}(y,t/|\mu|),\rho_0(y)\rangle
t^{-\alpha}dt\,.
\end{eqnarray*}
The Lemma follows from a Taylor expansion of
$A_{j\beta}^\sim(y,r,t,\mu^\prime)\phi^{(\beta)}(y,r)$ in powers of $r$. This expansion
is justified as follows. In Equation (\ref{eqn-2.k}), the integral $\bar d\zeta$ is
$O((1+s)^{-\infty})$ and so are its derivatives in $(y,r)$ in view of Equation (\ref{eqn-2.e}). Hence
$A_{j\beta}^\sim$ and its derivatives in $(y,r)$ are $O((1+t)^{-\infty})$. \end{proof}

For the term with $B_{j\beta}\phi^{(\beta)}$, we define
$$
d_{j\beta}^\sim(y,r,\xi,s,\mu)=\int_{-\infty}^\infty e^{-\sqrt{-1}s\zeta}d_{j\beta}(y,r,\xi,\zeta,\mu)
{\bar d\zeta}\,.
$$
One then has that
\begin{equation}\label{eqn-2.m}
B_{j\beta}(y,r,\mu)=\int_0^\infty d_{j\beta}^\sim(y,r,0,s,\mu)ds\,.
\end{equation}

\begin{lemma}\label{lem-2.5}
\begin{eqnarray*}
&&\int_0^\infty\langle B_{j\beta}\phi^{(\beta)}(y,r),\rho_0(y)\rangle r^{-\alpha}dr\\
&\sim&\sum_{k=0}^\infty|\mu|^{\alpha-1+j-k-|\beta|}\frac1{k!}
\int_0^\infty\int_0^\infty\langle d_{j\beta}^\sim(y,r,0,s,\mu^\prime)
\partial_r^k\phi^{(\beta)}(y,0),\rho_0(y)\rangle r^{k-\alpha}dsdr\,.
\end{eqnarray*}
The remainder after $N$ terms is bounded as in Lemma \ref{lem-2.4}.
\end{lemma}

\begin{proof} From Equation (\ref{eqn-2.g}), noting that
$d_{j\beta}=(-\sqrt{-1})^{|\beta|}\partial_{(\xi,\zeta)}^\beta d_j$, we have
$$
d_{j\beta}^\sim(y,r/t,t\xi,s/t,t\mu)=t^{j-|\beta|+1}d_{j\beta}^\sim(y,r,\xi,s,\mu)\,.
$$
We have from Equation (\ref{eqn-2.h}) for all $N$ that
$$(1+s^2)r^Nd_{j\beta}^\sim(y,r,0,s,\mu^\prime)=O(1)\,.$$
The Lemma now follows from a Taylor expansion of $\phi^{(\beta)}$ in powers of $r$.\end{proof}

Finally, since $c_{j\beta}$ is homogeneous of degree $j-|\beta|$ in $(\xi,\zeta,\mu)$, we have from Lemma
\ref{lem-2.2} that:
\begin{eqnarray}
&&\int_0^\infty\langle C_{j\beta}(y,r,\mu)\phi^{(\beta)}(y,r),\rho_0(y)\rangle r^{-\alpha}dr\nonumber\\
&=&|\mu|^{j-|\beta|}\int_0^\infty\langle
c_{j\beta}(y,r,0,0,\mu^\prime)\phi^{(\beta)}(y,r),\rho_0(y)\rangle r^{-\alpha}dr\,.\label{eqn-2.n}
\end{eqnarray}

We may add up all the expansions in Equation (\ref{eqn-2.n}), in Lemma \ref{lem-2.4}, and in Lemma
\ref{lem-2.5}. All terms are analytic for $\operatorname{Re}(\alpha)<1$ and we extend the expansion
by a meromorphic continuation to $\operatorname{Re}(\alpha)<2$ with a simple pole at $\alpha=1$. The terms
with a singularity at
$\alpha=1$ are
\begin{eqnarray*}
&&|\mu|^{j-|\beta|}\int_0^\infty\langle c_{j\beta}(y,r,0,0,\mu^\prime)\phi^{(\beta)}(y,r),
\rho_0(y)\rangle r^{-\alpha}dr\\
&+&|\mu|^{\alpha-1+j-|\beta|}\int_0^\infty\langle
A_{j\beta}^\sim(y,0,t,\mu^\prime)\phi^\beta(y,0),\rho_0(y)\rangle t^{-\alpha}dt\\
&+&|\mu|^{\alpha-1+j-|\beta|}
\int_0^\infty\int_0^\infty\langle d_{j\beta}^\sim(y,r,0,s,\mu^\prime)\phi^{(\beta)}(y,0),
\rho_0(y)\rangle r^{-\alpha} dsdr\,.
\end{eqnarray*}

We must now determine the extension to $\alpha=1$. As previously, when we regularize the
interior integral, we investigate
\begin{equation}\label{eqn-2.o}
\frac{c_{j\beta}(y,{0},0,0,\mu^\prime)+{|\mu|^{\alpha-1}}
\{A_{j\beta}^\sim(y,0,0,\mu^\prime)+\int_0^\infty d_{j\beta}^\sim(y,0,0,s,\mu^\prime)ds\}
}{1-\alpha}\,.
\end{equation}

From Equation (\ref{eqn-2.L}) with $|\mu|=1$, we have
$A_{j\beta}^\sim(y,0,0,\mu^\prime)=A_{j\beta}(y,0,\mu^\prime)$. From
Equation (\ref{eqn-2.m}),
$$\int_0^\infty
d_{j\beta}^\sim(y,0,0,s,\mu^\prime)ds=B_{j\beta}(y,0,\mu^\prime)\,.$$
Then from Equation (\ref{eqn-2.j}) and Lemma \ref{lem-2.2}, the expression in Equation (\ref{eqn-2.o})
is
$$c_{j\beta}(y,0,0,0,\mu^\prime)\frac{1-{|\mu|}^{\alpha-1}}{1-\alpha}\quad\text{for}\quad\alpha\ne1\,.$$
Taking the limit as $\alpha\rightarrow1$ yields
$$c_{j\beta}(y,0,0,0,\mu^\prime)\ln|\mu|,\quad\alpha=1$$
and consequently we have:

\begin{lemma}\label{lem-2.6}
For the parametrix
$P_N=\sum_{-j=2}^N[\Op(c_j)+\Op^\prime(d_j)]$, for $\phi\in\mathcal{S}(\mathbb{R}^m_+)$, for
$\rho(y,r)=\rho_0(y)r^{-\alpha}$, for $\operatorname{Re}(\alpha)<2$, and for $\alpha\ne1$, we have
\begin{eqnarray*}
&&\langle P_N\phi,{\rho}\rangle\sim
\sum_{j,\beta}|\mu|^{j-|\beta|}\frac1{\beta!}
\int\int_0^\infty\langle c_{j\beta}(y,{r},0,0,\mu^\prime){\phi^{(\beta)}(y,r),\rho_0(y)\rangle}
r^{-\alpha}drdy\\ &+&\sum_{j,\beta,k}|\mu|^{\alpha-1+j-|\beta|-k}
\frac1{\beta!k!}\int\int_0^\infty\partial_r^k\langle
A_{j\beta}^\sim\phi^{(\beta)},\rho_0\rangle(y,0,s,\mu^\prime)s^{k-\alpha}dsdy\\
&+&\sum_{j,\beta,k}|\mu|^{\alpha-1+j-|\beta|-k}
\frac1{\beta!k!}\int\int_0^\infty\int_0^\infty d_{j\beta}^\sim(y,r,0,s,\mu^\prime)r^{k-\alpha}drds\\
&&\qquad\cdot\langle\partial_r^k\phi^{(\beta)}(y,0),\rho_0(y)\rangle dy\,.
\end{eqnarray*}
When $\operatorname{Re}(\alpha)\ge1$, the divergent integrals are regularized as was discussed previously; when
$\alpha=1$ there are additional terms
$$\sum_{j,\beta}|\mu|^{j-|\beta|}\ln|\mu|\frac1{\beta!}
\int\langle c_{j\beta}(y,{0},0,0,\mu^\prime)\phi^{(\beta)}(y,0),\rho_0(y)\rangle dy\,.$$
\end{lemma}

The expansion in Lemma \ref{lem-2.6} is valid with $|\mu|$ replaced by $\mu$ and
$\mu^\prime$ replaced by $1$ in view of the following result:

\begin{lemma}\label{lem-2.7}
If $f(\mu)$ is holomorphic in a sector $\mathcal{S}$ which contains the positive real axis and if
$f(\mu)=|\mu|^jg(\mu^\prime)+{o}(|\mu|^j|)$, then $f(\mu)=g(1)\mu^j+{o}(|\mu|^j)$.
\end{lemma}

\begin{proof} Let $C$ be any closed curve in the sector $\mathcal{S}$. Then for $t>0$,
\begin{eqnarray*}
0&=&t^{-j-1}\int_{tC}f(\mu)d\mu=t^{-j-1}\int_{tC}\left\{|\mu|^jg(\mu^\prime)+{o}(|\mu|^j)\right\}d\mu\\
&=&\int_C\left\{|z|^jg(z^\prime)+{{o}(1)}\right\}dz\,.
\end{eqnarray*}
Let $t\rightarrow\infty$. By Morera's theorem, $|z|^jg(z^\prime)$ is holomorphic in $\mathcal{S}$
so it equals $z^jg(1)$, the holomorphic extension of its value on $\{z>0\}$.\end{proof}

\subsection{The case of a manifold $M$}\label{sect-2.3}
The proof of Theorem \ref{thm-2.1} now follows by standard
arguments \cite{Se69a, Gru96} involving a parametrix $P_N$ on the
manifold $M$ constructed from the Euclidean parametrices as in
Lemma \ref{lem-2.6}. In showing that
\begin{equation}\label{eqn-2.p}
\langle P_N\phi-(D_B+\mu^2)^{-1}\phi,{\rho}\rangle=O(\mu^{-K})
\end{equation}
for large $K$, we need to deal with the singularity $r^{-\alpha}$ in the specific heat. To this end, let
$R_N(y,r,\tilde y,s,\mu)$ be the kernel of $(D_B+\mu^2)^{-1}-P_N$. Then $R_N$ and its first
derivatives are $O(\mu^{-k})$ for large $K$. Moreover, by construction, the kernel of $P_N$ is zero when
$r=0$ so the same is true of $R_N$. It follows that $R_N$ is $O(r\mu^{-K})$ for large $K$ and Equation
(\ref{eqn-2.p}) follows.

\section{Heat content asymptotics for Dirichlet boundary conditions}\label{sect-3}

We adopt the notation of Theorem \ref{thm-1.4} throughout this section. Let $B=B_{\mathcal{D}}$ define Dirichlet boundary conditions. We
begin by using dimensional analysis to express the invariants
$\beta_{k,\alpha}^{\partial M}$ in terms of a Weyl basis of invariants which is formed by contracting indices.

\begin{lemma}\label{lem-3.1}
There exist universal constants $\varepsilon_\alpha^i$  so
that:
\medbreak$
\int_{\partial M}\beta_{0,\alpha}^{\partial M}(\phi,\rho,D,{B_{\mathcal{D}}})dy=\int_{\partial
M}\varepsilon_\alpha^0\langle\phi_0,\rho_0\rangle dy$,
\smallbreak$
\int_{\partial M}\beta_{1,\alpha}^{\partial M}(\phi,\rho,D,{B_{\mathcal{D}}})dy=\int_{\partial M}\left\{
\varepsilon_\alpha^1\langle\phi_1,\rho_0\rangle+\varepsilon_\alpha^2
\langle L_{aa}\phi_0,\rho_0\rangle+\varepsilon^3_\alpha\langle\phi_0,\rho_1\rangle\right\}$,
\smallbreak$
\int_{\partial M}\beta_{2,\alpha}^{\partial M}(\phi,\rho,D,{B_{\mathcal{D}}})dy
    =\int_{\partial M}\{\varepsilon_\alpha^4\langle\phi_2,\rho_0\rangle+\varepsilon_\alpha^5\langle L_{aa}\phi_1,\rho_0\rangle
    +\varepsilon_\alpha^6\langle E\phi_0,\rho_0\rangle$
\smallbreak$
\quad+\varepsilon_\alpha^7\langle\phi_0,\rho_2\rangle
   +\varepsilon_\alpha^8\langle L_{aa}\phi_0,\rho_1\rangle
    +\varepsilon_\alpha^9\langle\operatorname{Ric}_{mm}\phi_0,\rho_0\rangle
+\varepsilon_\alpha^{10}\langle L_{aa}L_{bb}\phi_0,\rho_0\rangle$
\smallbreak$
\quad+\varepsilon_\alpha^{11}\langle
L_{ab}L_{ab}\phi_0,\rho_0\rangle+\varepsilon_\alpha^{12}\langle \phi_{0;a},\rho_{0;a}\rangle
   +\varepsilon_\alpha^{13}\langle\tau\phi_0,\rho_0\rangle+\varepsilon_\alpha^{14}\langle\phi_1,\rho_1\rangle\}dy$,\\
    where $\tau$ is the scalar curvature.
\end{lemma}

\begin{proof} Let $c>0$. We consider a metric $g_c:=c^2g$. We then have for $\alpha\not=1$ (see the paragraph after \eqref{eqn-3.b} for the case $\alpha=1$):
\begin{equation}\label{eqn-3.a}
\begin{array}{lll}
dx_c:=c^mdx,&dy_c:=c^{m-1}dy,&D_c:=c^{-2}D,\\
r_c:=cr,&\partial_{r_c}=c^{-1}\partial_r,&\phi_{i,c}=c^{\alpha-i}\phi_i,\\
{\mathcal{I}_{\operatorname{Reg},c}}=c^m{\mathcal{I}_{\operatorname{Reg}}},&\rho_{i,c}=c^{-i}\rho_i\,.
\end{array}\end{equation}
We then compute:
\begin{eqnarray*}
&&\beta(\phi,\rho,D_c,B_{\mathcal{D}})(t)=\int_M\langle e^{-tD_{c,B}}\phi,\rho\rangle dx_c
\\&=&c^m\int_M\langle e^{c^{-2}tD_B}\phi,\rho\rangle dx
=c^m\beta(\phi,\rho,D,{B_{\mathcal{D}}})(c^{-2}t)\,.
\end{eqnarray*}
{We take $\alpha\notin\mathbb{Z}$ and then continue analytically to the integer values with $\alpha\ne1$. The interior and
boundary terms then decouple so we may conclude:}
\begin{eqnarray*}
&&\sum_{k=0}^\infty t^{(1+k-\alpha)/2}\int_{\partial M}
    \sum_{i+j\le k}\beta_{k,\alpha,i,j}^{\partial M}(c^{\alpha-i}\phi_i,c^{-j}\rho_j,c^{-2}D,{B_{\mathcal{D}}})c^{m-1}dy\\
&=&c^m\sum_{k=0}^\infty c^{\alpha-k-1}t^{(1+k-\alpha)/2}\int_{\partial M}\sum_{i+j\le k}
    \beta_{k,\alpha,i,j}^{\partial M}(\phi_i,\rho_j,D,{B_{\mathcal{D}}})dy\,.
\end{eqnarray*}
Equating powers of $t$ in the asymptotic series and simplifying yields:
$$
\int_{\partial M}\beta_{k,\alpha,i,j}^{\partial M}(\phi_i,\rho_j,c^{-2}D,B_{\mathcal{D}})dy
=c^{i+j-k}\int_{\partial M}\beta_{k,\alpha,i,j}^{\partial M}(\phi_i,\rho_j,D,B_{\mathcal{D}})dy\,.
$$

Studying relations of this kind is by now quite standard and we refer to \cite{G04} for further details.
We conclude that
$\beta_{k,\alpha,i,j}^{\partial M}$ is homogeneous of weighted degree
$k$ in the jets of $\phi_i$ and of $\rho_j$. We now use Weyl's theory of invariants to write down a
spanning set for the invariants which arise in this way. There is some indeterminacy in these invariants as we can
always integrate by parts to eliminate tangential divergence terms. For example, we have
\begin{equation}\label{eqn-3.b}
\int_{\partial M}\langle\phi_{0:aa},\rho_0\rangle dy=\int_{\partial M}\langle\phi_0,\rho_{0:aa}\rangle dy
=-\int_{\partial M}\langle\phi_{0:a},\rho_{:a}\rangle dy\,.
\end{equation}
For this reason, we have eliminated the first two invariants from the formula in Lemma \ref{lem-3.1}. This completes the proof of Lemma
\ref{lem-3.1} for $\alpha\ne1$.

If $\alpha=1$, the argument is different. In this instance, the regularizing term is given by:
\begin{eqnarray*}
&&\ln(\varepsilon_c)\int_{\partial M}\langle\phi_{0,c},\rho_{0,c}\rangle dy_c
=\ln(c\varepsilon)\int_{\partial M}c\langle\phi_0,\rho_0\rangle c^{m-1}dy\\
&=&c^m\left\{\ln(\varepsilon)+\ln(c)\right\}\int_{\partial M}\langle\phi_0,\rho_0\rangle dy\,.
\end{eqnarray*}
This then yields the modified relation:
\begin{equation}\label{eqn-3.c}
\mathcal{I}_{\operatorname{Reg},c}(\phi,\rho)=c^m\mathcal{I}_{\operatorname{Reg}}(\phi,\rho)+\ln(c)c^m\int_{\partial
M}\langle\phi_0,\rho_0\rangle dy\,.
\end{equation}

Recall the notation of Theorem \ref{thm-1.4}. One has that:
\begin{eqnarray*}
&&\beta(\phi,\rho,c^{-2}D,B_{\mathcal{D}})(t)
\sim\sum_{n=0}^\infty\frac{(-t)^n}{n!}\mathcal{I}_{\operatorname{Reg},c}(\phi,c^{-2n}\tilde D^n\rho)\\
&+&\sum_{k=0}^\infty t^{k/2}\ln(t)\int_{\partial
M}\check\beta_{k}(\phi,\rho,c^{-2}D,B_{\mathcal{D}})c^{m-1}dy\\
&+&\sum_{k=0}^\infty\sum_{i+j\le k} t^{k/2}\int_{\partial M}\beta_{k,1,i,j}^{\partial
M}(c^{1-i}\phi_i,c^{-j}\rho_j,c^{-2}D,B_{\mathcal{D}}) c^{m-1}dy\\
&=&c^m\beta(\phi,\rho,D,B_{\mathcal{D}})(c^{-2}t)
\sim\sum_{n=0}^\infty c^{m-2n}\frac{(-t)^n}{n!}\mathcal{I}_{\operatorname{Reg}}(\phi,\tilde D^n\rho)\\
&+&\sum_{k=0}^\infty c^{m-k}t^{k/2}(\ln(t)-2\ln(c))\int_{\partial
M}\check\beta_{k}(\phi,\rho,D,B_{\mathcal{D}})dy\\
&+&\sum_{ k=0}^\infty\sum_{i+j\le k} c^{m-k}t^{k/2}\int_{\partial M}\beta_{k,1,i,j}^{\partial M}(\phi_i,\rho_j,D,B_{\mathcal{D}})
dy\,.
\end{eqnarray*}
Equating terms in the asymptotic expansion then yields:
\begin{eqnarray}\label{eqn-3.d}
&&\int_{\partial M}\beta_{k,1,i,j}^{\partial M}(\phi_i,\rho_j,c^{-2}D,B_{\mathcal{D}})dy
=c^{i+j-k}\int_{\partial M}\beta_{k,1,i,j}^{\partial M}(\phi_i,\rho_j,D,B_{\mathcal{D}})dy,\\
&&\frac{(-1)^n}{n!}\mathcal{I}_{\operatorname{Reg,c}}(\phi,\tilde D^n\rho)=c^m\left\{
\mathcal{I}_{\operatorname{Reg}}(\phi,\tilde D^n\rho)
-2\ln(c)\int_{\partial M}\tilde\beta_{2n}(\phi,\rho,D,B_{\mathcal{D}})dy\right\},\nonumber\\
&&0=-2\ln(c)\int_{\partial M}\tilde\beta_{2n+1}(\phi,\rho,D,B_{\mathcal{D}})dy
\,.\nonumber
\end{eqnarray}
The weighted homogeneity of $\beta_{k,1,i,j}^{\partial M}$ now follows and completes the proof of Lemma \ref{lem-3.1}.
Theorem \ref{thm-1.4} (4) follows from Equations (\ref{eqn-3.c}) and (\ref{eqn-3.d})
\end{proof}

We shall prove Assertions (1) and (2) of Theorem \ref{thm-1.5} by evaluating the normalizing constants in Lemma \ref{lem-3.1}. We begin by
establishing some product formulas:

\begin{lemma}\label{lem-3.2}
 Suppose that $M=M_1\times M_2$, that $g_M=g_{M_1}+g_{M_2}$, that $\partial {M_1}=\emptyset$, and
that
$D_M=D_{M_1}+D_{M_2}$ where $D_{M_1}$ and $D_{M_2}$ are scalar operators of Laplace type on ${M_1}$ and on ${M_2}$, respectively. Suppose
that
$\phi_M=\phi_{M_1}\phi_{M_2}$ and $\rho_M=\rho_{M_1}\rho_{M_2}$ decompose similarly. Then
\begin{enumerate}
\item
$\beta(\phi_M,\rho_M,D_M,B_{\mathcal{D}})(t)=\beta(\phi_{M_1},\rho_{M_1},D_{M_1})(t)\cdot\beta(\phi_{M_2},\rho_{M_2},D_{M_2},B_{\mathcal{D}})(t)$.
\item $\int_{\partial M}\beta_{k,\alpha}^{\partial M}(\phi_M,\rho_M,D_M,B_{\mathcal{D}})dy$\smallbreak$=
\sum_{2n+j=k}\frac{(-1)^n}{n!}\int_{M_1}\langle\phi_{M_1},(\tilde D_{M_1})^n\rho_{M_1}\rangle dx_{M_1}$\smallbreak$\quad\times
\int_{\partial M_2}\beta_{j,\alpha}^{\partial M_2}(\phi_{M_2},\rho_{M_2},D_{M_2},B_{\mathcal{D}})dy_{M_2}$.
\item The universal constants $\varepsilon_\alpha^i$ are dimension free.
\item $\varepsilon_\alpha^6=\varepsilon_\alpha^0$, $\varepsilon_\alpha^{13}=0$, and $\varepsilon_\alpha^{12}=-\varepsilon_\alpha^0$.
\end{enumerate}\end{lemma}

\begin{proof} Assertion (1) follows from the identity $e^{-tD_{M,B}}=e^{-tD_{M_1}}e^{-tD_{{M_2},B}}$ and Assertion (2) follows from
Assertion (1). If we take $M_1=S^1$, $D_{M_1}=-\partial_\theta^2$, $\phi_{M_1}=1$, and $\rho_{M_1}=1$, we have that
$\beta(\phi_{M_1},\rho_{M_1},D_{M_1})(t)=2\pi$. This then yields the identity
$$\int_{\partial M}\beta_{k,\alpha}^{\partial M}(\phi_{M_2},\rho_{M_2},D,B_{\mathcal{D}})dy
=2\pi\int_{\partial M_2}\beta_{k,\alpha}^{\partial
M_2}(\phi_{M_2},\rho_{M_2},D_{M_2},B_{\mathcal{D}})dy_2\,.$$ Assertion (3) now follows.

We take $M_2=[0,1]$ and
$D_2=-\partial_r^2$. We take 
\begin{eqnarray*}
&&\phi_{M_2}=\rho_{M_2}=0\quad\text{near}\quad r=1,\\
&&\rho_{M_2}=1\quad\text{and}\quad\phi_{M_2}=r^{-\alpha}\quad\text{near}\quad r=0\,.
\end{eqnarray*}
Since the structures on $M_2$ are flat,
$$
\beta_k^{\partial M_2}(\phi_{M_2},\rho_{M_2},D_{M_2},B_{\mathcal{D}})(r)=
\left\{\begin{array}{lll}
0&\text{if}&r=1\quad\text{and}\quad k\ge0,\\
0&\text{if}&r=0\quad\text{and}\quad k>0,\\
\varepsilon_\alpha^0&\text{if}&r=0\quad\text{and}\quad k=0.
\end{array}\right.$$
As the second fundamental form vanishes, the distinction between `;' and `:' disapears and we
may use Eauation (\ref{eqn-1.c}) to see that
$\tilde D_1\rho_{M_1}=-(\rho_{;aa}+\tilde E\rho)$.
Theorem \ref{thm-1.4} then implies
$$
\beta_{2}(\phi_{M_1},\rho_{M_1},D_{M_1})=\int_{\partial M}\langle\phi_{M_1},\rho_{M_1;aa}+\tilde
E\rho_{M_1}\rangle dx_1\,.
$$
We may therefore use Assertion (2) to see
\begin{eqnarray*}
&&\int_{\partial M}\beta_{2,\alpha}^{\partial M}(\phi_M,\rho_M,D_M,B_{\mathcal{D}})dy
=\varepsilon_\alpha^0\int_{M_1}\langle\phi_{M_1},\rho_{M_1;aa}+\tilde E\rho_{M_1}\rangle dx_1
\end{eqnarray*}
Assertion
(4) now follows from this identity.
\end{proof}

Next, we evaluate the universal constants $\varepsilon_\alpha^0$.

\begin{lemma}\label{lem-3.3}
\ \begin{enumerate}
\item If $\alpha\ne1$, then $\varepsilon_\alpha^0=\pi^{-1/2}2^{1-\alpha}\Gamma\left(\frac{2-\alpha}2\right)(\alpha-1)^{-1}$.
\item Let $\gamma$ be Euler's constant. Then $\varepsilon_1^0=\frac\gamma2$.
\end{enumerate}\end{lemma}

\begin{proof} The proof follows from \eqref{eqn-2.d1}.
We note that Assertion (1) also follows for $1<\alpha<2$ by the special case calculation in \cite{Be07}. Assertion (1)
then follows for $\alpha\not=1$ by analytic continuation.

To study the case $\alpha=1$ by a special case calculation we let
$M=[0,\infty)$, let $\Theta=1$ on $[0,\varepsilon]$ and with
compact support in $[0,1)$, let $\phi=r^{-1}\Theta(r)$, let
$\rho=1$, let $D=-\partial_r^2$, and let $\gamma$ be Euler's
constant.
\def\grac#1#2{{\textstyle\frac{#1}{#2}}}As is usual, we work dually and compute $\beta(\rho,\phi,D,B_{\mathcal{D}})(t)$. The halfspace
solution of the heat equation with constant initial temperature is given by:
$$u(r;t)=\grac{2}{\sqrt\pi}\int_0^{\frac r{2\sqrt t}}e^{-s^2}ds\,.$$
Consequently, we may compute: \medbreak$\displaystyle
\beta(\rho,\phi,D,B_{\mathcal{D}})(t)=\grac2{\sqrt\pi}\int_0^1\int_0^{\frac
r{2\sqrt t}}e^{-s^2}r^{-1}\Theta(r)dsdr$
\smallbreak\quad$\displaystyle=\frac2{\sqrt\pi}\int_0^1\ln(r)\partial_r\left\{\Theta(r)\int_0^{\frac
r{2\sqrt t}}e^{-s^2}ds\right\}dr$
\smallbreak\quad$\displaystyle=-\frac1{\sqrt{\pi
t}}\int_0^1\ln(r)\Theta(r)e^{-\frac{r^2}{4t}}dt$
$-\displaystyle\frac2{\sqrt\pi}\int_\varepsilon^1\int_0^{\frac
r{2\sqrt t}}e^{-s^2}\ln(r)\Theta^\prime(r)dsdr$
\smallbreak\quad$=B_1+B_2,$\medbreak\noindent where \medbreak
$\displaystyle B_1=-\grac1{\sqrt\pi t}\left\{\int_0^1\ln(r)
e^{-\frac{r^2}{4t}}dr+\int_{\varepsilon}^1
\ln(r)(\Theta(r)-1)e^{-\frac{r^2}{4t}}dr\right\}$
\smallbreak\quad$\displaystyle=-\grac1{\sqrt\pi
t}\int_0^\infty\ln(r)e^{-\frac{r^2}{4t}}dr+O(e^{-\frac{\varepsilon^2}{8t}})$
\smallbreak\quad$\displaystyle=-\grac2{\sqrt\pi}\left\{\int_0^\infty\ln(r)e^{-r^2}dr+\int_0^\infty\ln(2\sqrt
t)e^{-r^2}dr\right\} +O(e^{-\frac{\varepsilon^2}{8t}})$
\smallbreak\quad$\displaystyle=-\grac12\ln(t)+\left\{-\ln(2)-\grac2{\sqrt\pi}\int_0^\infty\ln(r)e^{-r^2}dr\right\}
+O(e^{-\frac{\varepsilon^2}{8t}}),$ \smallbreak $\displaystyle
B_2=-\grac2{\sqrt\pi}\int_{\varepsilon}^1\ln(r)\Theta^\prime(r)\left\{
\int_0^\infty e^{-s^2}ds-\int_{\frac r{2\sqrt t}}^\infty
e^{-s^2}ds\right\}dr$
\smallbreak\quad$\displaystyle=-\int_{\varepsilon}^1\ln(r)\Theta^\prime(r)dr+O(e^{-\frac{\varepsilon^2}{8t}})$
\smallbreak\quad$\displaystyle=-\ln(r)\Theta(r)\bigg|_{\varepsilon}^1+\int_{\varepsilon}^1
r^{-1}\Theta(r)dr +O(e^{-\frac{\varepsilon^2}{8t}})$
\smallbreak\quad$\displaystyle=\ln(\varepsilon)+\int_{\varepsilon}^1
r^{-1}\Theta(r)dr+O(e^{-\frac{\varepsilon^2}{8t}})$.
\medbreak\noindent This then yields the expression
\medbreak$\displaystyle\beta(\rho,\phi,D,B_{\mathcal{D}})(t)$
\medbreak$\quad\displaystyle=\grac12\ln(\grac{\varepsilon^2}t)
-\ln(2)-\grac2{\sqrt\pi}\int_0^\infty\ln(s)e^{-s^2}ds+\int_{\varepsilon}^1
r^{-1}\Theta(r)dr+O(e^{-\frac{\varepsilon^2}{8t}})$.
\medbreak\noindent Since $\phi$ is compactly supported in $[0,1)$,
the heat content  for corresponding problem on the interval
$[0,1]$ is the same as for $[0,\infty)$ up to $O(t^\infty)$.
Assertion (2) now follows.
\end{proof}

We continue our study by index shifting:
\begin{lemma}\label{lem-3.4}
\ \begin{enumerate}
\item Assume $\nabla_{\partial_r}\Phi=0$ on ${\mathcal{C}_\epsilon}$. Set
$\phi:={\chi(r)}\Phi(y)r^{i_0-\gamma}$. Then
$$\textstyle\int_{\partial M}{\beta_{k,\gamma,i_0,j}^{\partial
M}}(\Phi,\rho_j,D,{B_{\mathcal{D}}})dy =\int_{\partial M}{\beta_{k-i_0,\gamma-i_0,0,j}^{\partial
M}}(\Phi,\rho_j,D,{B_{\mathcal{D}}})dy\,.$$
\item If $\alpha\ne1$, then $\varepsilon_\alpha^1=\varepsilon_{\alpha-1}^0$, $\varepsilon_\alpha^4=\varepsilon_{\alpha-2}^0$,
$\varepsilon_\alpha^5=\varepsilon_{\alpha-1}^2$, and $\varepsilon_\alpha^{14}=\varepsilon_{\alpha-1}^3$.
\item $\varepsilon_1^1=-\frac2{\sqrt\pi}$, $\varepsilon_1^4=-1$, $\varepsilon_1^5=\frac12$, and $\varepsilon_1^{14}=0$.
\end{enumerate}\end{lemma}

\begin{proof} Set $\phi:={\chi(r)}\Phi(y)r^{i_0-\gamma}$. By Theorem \ref{thm-1.4}
 with $\alpha=\gamma$ and with {$\alpha=\gamma-i_0$},
\begin{eqnarray*}
&&\sum_{i_0+j\le k}t^{(1+k-\gamma)/2}\int_{\partial M}\beta_{k,\gamma,i_0,j}(\Phi,\rho_j,D,{B_{\mathcal{D}}})dy\\
&\sim&\sum_{j\le\ell}t^{(1+\ell-(\gamma-{i_0}))/2}\int_{\partial
M}\beta_{\ell,\gamma-{i_0},0,j}(\Phi,\rho_j,D,{B_{\mathcal{D}}})dy\,.
\end{eqnarray*}
We set $k=\ell+{i_0}$ and equate powers of $t$ to establish Assertion (1); Assertion (2) then follows from Lemma \ref{lem-3.1}
and Assertion (3) from Theorem \ref{thm-1.3}.
\end{proof}

We continue our study with a functorial property that exploits the
fact that we are working in a very general context; we are no
longer working with the scalar Laplacian! Even if one {were} only
interested in the scalar Laplacian, it would be necessary to
consider general operators of Laplace type in order to use this
functorial property! Let $\Theta=1$ on $[0,\frac12]$ and with
compact support in $[0,1)$.

\begin{lemma}\label{lem-3.5}
Let $\mathbb{T}^{m-1}$ be the torus with periodic parameters
$(y_1,...,y_{m-1})$. Let $M=\mathbb{T}^{m-1}\times[0,1]$.
Let $f_a\in C^\infty([0,1])$ satisfy $f_a(0)=0$ and $f_a\equiv0$ near $r=1$. Let $\delta_a\in\mathbb{R}$. Set
$$\begin{array}{ll}
ds^2_M=\textstyle\sum_ae^{2f_a(r)}dy_a\circ dy_a+dr\circ dr,&
\rho:=e^{-\sum_af_a(r)},\\
D_M:=-\textstyle\sum_ae^{-2f_a(r)}(\partial_{y_a}^2+\delta_a\partial_{y_a})-\partial_r^2,
&\phi:=\Theta(r)r^{-\alpha}\,.\vphantom{\vrule height 11pt}
\end{array}$$
\begin{enumerate}
\item If $k>0$, then {$\int_{\partial M}\beta_{k,\alpha}^{\partial M}(\phi,\rho,D_M,{B_{\mathcal{D}}})dy=0$}.
\item $-\frac12\varepsilon_\alpha^1-\varepsilon_\alpha^2=0$.
\item $-\frac14(\varepsilon_\alpha^6+\varepsilon_\alpha^{12})=0$.
\item $-\frac14\varepsilon_\alpha^4+\frac12\varepsilon_\alpha^6-\frac14\varepsilon_\alpha^7-\varepsilon_\alpha^9=0$.
\item
$\frac18\varepsilon_\alpha^4+\frac12\varepsilon_\alpha^5+\frac14\varepsilon_\alpha^6+\frac18\varepsilon_\alpha^7
+\frac12\varepsilon_\alpha^{8}
+\varepsilon_\alpha^{10}=0$.
\item $\textstyle-\varepsilon_\alpha^9+\varepsilon_\alpha^{11}=0$.
\end{enumerate}
\end{lemma}

\begin{proof} We use $-\partial_r^2$ on $[0,1]$ and
$D_M$ on $M$. Since $\phi$ vanishes near $r=1$, this boundary component plays no role. Let
$u_B(r;t)$ be the solution of the heat equation on
$[0,1]$ with Dirichlet boundary conditions. Since the problem decouples, $u_B(r;t)$ is
also the solution of the heat equation on
$M$ with Dirichlet boundary conditions. The Riemannian measure
$$dx=\sqrt{\det g_{ij}}dydr=e^{\sum_af_a}dydr\,.$$
As $\rho=e^{-\sum_af_a}$, $\rho dx=dydr$. We suppose $\alpha\ne1$. Since
$\operatorname{vol}(\mathbb{T}^{m-1})=(2\pi)^{m-1}$,
\begin{eqnarray*}
&&\beta(\phi,\rho,D,{B_{\mathcal{D}}})(t)=
\int u_B(r;t)\rho dx=(2\pi)^{m-1}\int_0^1u_B(r;t)dr\\
&=&(2\pi)^{m-1}\beta(\phi,1,-\partial_r^2)(t)=(2\pi)^{m-1}(\mathcal{I}_{\operatorname{Reg}}(\phi,1)+\varepsilon_\alpha^0)+O(t^n)
\end{eqnarray*}
for any $n$ since the structures are flat on the interval. Note that $\tilde D\rho=0$. Assertion (1) now follows. If $\alpha=1$, the
computation must be modified to take the $\ln$ term into account; this does not affect the computation of $\beta_k^{\partial M}$
for $k>1$ and the desired conclusion follows similarly.

To apply Assertion (1), we must determine the relevant tensors. The formalism of Lemma \ref{lem-1.1} is crucial at this point as the
connection defined by the operator $D_M$ is no longer flat. We have:
$$
\begin{array}{ll}
\textstyle\omega_a=\frac12\delta_a,&\tilde\omega_a=-\omega_a=-\frac12\delta_a,\\
\omega_m=-\frac12\sum_af_a^\prime,&
\tilde\omega_m=-\omega_m=\frac12\sum_af_a^\prime\,.\vphantom{\vrule height 10pt}
\end{array}$$
We compute:
\begin{equation}\label{eqn-3.e}
\begin{array}{rl}
\phi_0=&1,\\
\phi_1=&\{\nabla_{\partial r}(r^\alpha\phi)\}|_{\partial M}=\{(\partial_r-\textstyle\frac12\sum_af_a^\prime)(1)\}|_{\partial M}=
-\frac12\sum_af_a^\prime(0),\vphantom{\vrule height 11pt}\\
\textstyle\phi_2=&\frac12\{(\nabla_{\partial_r})^2(r^\alpha\phi)\}|_{\partial M}
=\frac12\{(\partial_r-\textstyle\frac12\sum_af_a^\prime)^2(1)\}|_{\partial M}\vphantom{\vrule height 11pt}\\
\textstyle=&\frac18(\sum_af_a^\prime(0))^2-\frac14\sum_af_a^{\prime\prime}(0),\vphantom{\vrule height 11pt}\\
\rho_0=&1,\vphantom{\vrule height 11pt}\\
\textstyle\rho_1=&\{\tilde\nabla_{\partial r}(\rho)\}|_{\partial M}
=\{(\partial_r+\frac12\sum_af_a^\prime)(e^{-\sum_af_a})\}|_{\partial
M}=-\frac12\sum_af_a^\prime(0),\vphantom{\vrule height 11pt}\\
\textstyle\rho_2=&\frac12\{(\tilde\nabla_{\partial r})^2\rho\}|_{\partial M}
=\frac12\{(\partial_r+\frac12\sum_af_a^\prime)^2(e^{-\sum_af_a})\}|_{\partial M}\vphantom{\vrule height 11pt}\\
=&\textstyle\frac18(\sum_af_a^\prime(0))^2-\frac14\sum_af_a^{\prime\prime}(0)\,.\vphantom{\vrule
height 11pt}
\end{array}\end{equation}
It is straightforward to compute that we have the following relations when $r=0$; we refer to Lemma 2.3.7
\cite{G04} for further details:
\begin{equation}\label{eqn-3.f}
\begin{array}{l}
E=\frac12\sum_af_a^{\prime\prime}+\frac14(\sum_af_a^\prime)^2-\frac14\sum_a\delta_a^2,\quad
   L_{aa}=-\sum_af_a^\prime,\\
\operatorname{Ric}_{mm}=-\sum_a((f_a^\prime)^2+f_a^{\prime\prime}),\quad
L_{aa}L_{bb}=(\sum_af_a^\prime)^2,\quad
L_{ab}L_{ab}=\sum_a(f_a^\prime)^2\,.\vphantom{\vrule height 11pt}
\end{array}\end{equation}

Considering the term $\sum_af_a^\prime$ in $\beta_{1,\alpha}^{\partial M}$ yields Assertion (2), considering the term $\sum_a\delta_a^2$ in
$\beta_{2,\alpha}^{\partial M}$ yields Assertion (3), considering the term $\sum_af_a^{\prime\prime}$ in $\beta_{2,\alpha}^{\partial M}$
yields Assertion (4), considering the term $(\sum_af_a^\prime)^2$ in $\beta_{2,\alpha}^{\partial M}$ yields Assertion
(5), and considering the term $\sum_a(f_a^\prime)^2$ in $\beta_{2,\alpha}^{\partial M}$ yields Assertion (6).
\end{proof}

We continue our discussion:

\begin{lemma}\label{lem-3.6}
\ \begin{enumerate}
\smallbreak\item If $\rho_0=0$, then $\partial_t\beta(\phi,\rho,D,{B_{\mathcal{D}}})(t)=-\beta(\phi,\tilde
D\rho,D,B_{\mathcal{D}})(t)$.
\item $\varepsilon_\alpha^3=0$.
\item If $\alpha\ne1$, $\varepsilon_\alpha^7=\frac{4}{3-\alpha}\varepsilon_\alpha^0$,
$\varepsilon_\alpha^8=-\frac2{3-\alpha}\varepsilon_\alpha^0$, and $\varepsilon_\alpha^{14}=0$.
\item $\varepsilon_1^7=2\varepsilon_1^0+1$, $\varepsilon_1^8=-\varepsilon_1^0-\frac12$, and $\varepsilon_\alpha^{14}=0$.
\end{enumerate}
\end{lemma}

\begin{proof}
Assume $\rho_0=0$. By Equation (\ref{eqn-1.c}) we have
\begin{equation}\label{eqn-3.g}
-(\tilde D\rho)_0=2\rho_2-L_{aa}\rho_1\,.
\end{equation}
We compute that:
\begin{eqnarray*}
\partial_t\beta(\phi,\rho,D,{B_{\mathcal{D}}})(t)&=&-\langle De^{-tD_B}\phi,\rho\rangle=
-\langle e^{-tD_B}\phi,\tilde D\rho\rangle\\ &=&-\beta(\phi,\tilde
D\rho,{D,B_{\mathcal{D}}})(t),
\end{eqnarray*}
where the middle equality is justified as $\rho_0=0$. This proves Assertion (1). We use Assertion (1) to see
$$\int_{\partial M}\beta_{1,\alpha}^{\partial M}(\phi,\rho,D,B_{\mathcal{D}})dy=0\,.$$
Since there is no restriction on $\rho_1$, we conclude that $\varepsilon_\alpha^3=0$ which establishes
Assertion (2). Furthermore, if $\alpha\ne1$, we may conclude
\begin{equation}\label{eqn-3.h}
{\textstyle\frac{1+k-\alpha}2}\int_{\partial M}\beta_{k,\alpha}^{\partial M}(\phi,\rho,D,B_{\mathcal{D}})dy
=-\int_{\partial M}\beta_{k-2,\alpha}^{\partial M}(\phi,\tilde D\rho,D,B_{\mathcal{D}})dy\,.
\end{equation}
We set $k=2$ to see that
\begin{eqnarray*}
&&-\int_{\partial M}\varepsilon_\alpha^0\langle\phi_0,(\tilde D\rho)_0\rangle\\
&=&{\textstyle\frac{3-\alpha}2}\int_{\partial M}\left\{\langle
\varepsilon_\alpha^7\langle\phi_0,\rho_2\rangle +\varepsilon_\alpha^8\langle L_{aa}\phi_0,\rho_1\rangle
+\varepsilon_\alpha^{14}\langle\phi_1,\rho_1\rangle\right\} dy\,.
\end{eqnarray*}
Assertion (3) now follows from Equation (\ref{eqn-3.g}). If $\alpha=1$, we have:
\begin{eqnarray*}
&&-\int_{\partial M}\varepsilon_1^0\langle\phi_0,(\tilde D\rho)_0\rangle\\
&=&\int_{\partial M}\left\{{\textstyle\frac12}\langle\phi_0,(\tilde D\rho)_0\rangle
+\varepsilon_1^7\langle\phi_0,\rho_2\rangle +\varepsilon_1^8\langle L_{aa}\phi_0,\rho_1\rangle
+\varepsilon_1^{14}\langle\phi_1,\rho_1\rangle\right\} dy\,.
\end{eqnarray*}
Assertion (4) now follows.
\end{proof}

Assertions (1) and (2) of Theorem \ref{thm-1.5} will follow from Lemma \ref{lem-3.2} and from the following result:
\begin{lemma}\label{lem-3.7}
\ \begin{enumerate}\item Suppose that $\alpha\ne1$. Then:
\begin{enumerate}
\item $\varepsilon_\alpha^0=c_\alpha$.
\item $\varepsilon_\alpha^1=c_{\alpha-1}$, $\varepsilon_\alpha^2=-\frac12c_{\alpha-1}$,
and $\varepsilon_\alpha^3=0$.
\item $\varepsilon_\alpha^4=c_{\alpha-2}$ and $\varepsilon_\alpha^5=-\frac12c_{\alpha-2}$.
\item $\varepsilon_\alpha^6=-\frac{\alpha-3}{2(\alpha-1)(\alpha-2)}c_{\alpha-2}$, $\varepsilon_\alpha^7=\frac2{(\alpha-1)(\alpha-2)}c_{\alpha-2}$ and
$\varepsilon_\alpha^8=-\frac1{(\alpha-1)(\alpha-2)}c_{\alpha-2}$.
\item $\varepsilon_\alpha^9=-\frac{\alpha-1}{4(\alpha-2)}c_{\alpha-2}$ and $\varepsilon_\alpha^{10}=\frac{\alpha-1}{8(\alpha-2)}c_{\alpha-2}$.
\item $\varepsilon_\alpha^{11}=-\frac{\alpha-1}{4(\alpha-2)}c_{\alpha-2}$,
$\varepsilon_\alpha^{12}=\frac{\alpha-3}{2(\alpha-1)(\alpha-2)}c_{\alpha-2}$,
$\varepsilon_\alpha^{13}=0$,  and
$\varepsilon_\alpha^{14}=0$.
\end{enumerate}
\item Let $\alpha=1$. Then:
\begin{enumerate}
\item $\varepsilon_1^0=\frac12\gamma$.
\item $\varepsilon_1^1=-\frac2{\sqrt\pi}$, $\varepsilon_2^1=\frac1{\sqrt\pi}$, and $\varepsilon_3^1=0$.
\item $\varepsilon_1^4=-1$, $\varepsilon_1^5=\frac12$.
\item $\varepsilon_1^6=\frac12\gamma$, $\varepsilon_1^7=\gamma+1$, and $\varepsilon_1^8=-\frac12\gamma-\frac12$.
\item $\varepsilon_1^9=0$ and $\varepsilon_1^{10}=0$.
\item $\varepsilon_1^{11}=0$, $\varepsilon_1^{12}=\frac12\gamma$, $\varepsilon_1^{13}=0$, and $\varepsilon_1^{14}=0$.
\end{enumerate}\end{enumerate}
\end{lemma}

\begin{proof} Let $\alpha\ne1$.
Assertions (1a) and (2a) follow from Lemma \ref{lem-3.3}. Assertions (1b) and (2b) follow from Lemmas \ref{lem-3.4}, \ref{lem-3.5}, and
\ref{lem-3.6}. Assertions (1c) and (2c) follow from Assertion (1b) and from Lemma \ref{lem-3.4}. Because
$s\Gamma(s)=\Gamma(s+1)$, we have:
$$
\textstyle c_\alpha=-\frac{\alpha-3}{2(\alpha-1)(\alpha-2)}c_{\alpha-2}\,.
$$
Assertions (1d) and (2d) now follow from Lemmas \ref{lem-3.2} and \ref{lem-3.6}.
We use Lemma \ref{lem-3.5}  to establish Assertions (1e) and (2e) by computing for
$\alpha\ne1$ that
\begin{eqnarray*}
\varepsilon_\alpha^9&=&\textstyle-\frac14\varepsilon_\alpha^4+\frac12\varepsilon_\alpha^6-\frac14\varepsilon_\alpha^7\\
&=&\textstyle c_{\alpha-2}(-\frac14-\frac14\frac{\alpha-3}{(\alpha-1)(\alpha-2)}
-\frac14\frac2{(\alpha-1)(\alpha-2)})=-\frac{\alpha-1}{4(\alpha-2)}c_{\alpha-2},\\
\varepsilon_\alpha^{10}&=&\textstyle-\frac18\varepsilon_\alpha^4-\frac12\varepsilon_\alpha^5
-\frac14\varepsilon_\alpha^6-\frac18\varepsilon_\alpha^7-\frac12\varepsilon_\alpha^8\\
&=&\textstyle c_{\alpha-2}(-\frac18+\frac14+\frac18\frac{\alpha-3}{(\alpha-1)(\alpha-2)}-\frac14\frac1{(\alpha-1)(\alpha-2)}
+\frac12\frac1{(\alpha-1)(\alpha-2)})=\frac{\alpha-1}{8(\alpha-2)}c_{\alpha-2}
\end{eqnarray*}
and for $\alpha=1$ that
\begin{eqnarray*}
&&\textstyle\varepsilon_1^9=-\frac14(-1)+\frac12(\frac12\gamma)-\frac14(\gamma+1)=0,\\
&&\textstyle\varepsilon_1^{10}=-\frac18(-1)-\frac12(\frac12)-\frac14(\frac12\gamma)-\frac18(\gamma+1)-\frac12(-\frac12\gamma-\frac12)=0\,.
\end{eqnarray*}
Assertions (1f) and (2f) follow from Lemmas \ref{lem-3.5}, \ref{lem-3.2}, and \ref{lem-3.6}.
\end{proof}

\section{Heat content asymptotics for Robin boundary conditions}\label{sect-4}
Section \ref{sect-4} is devoted to the proof of Assertion (3) of Theorem \ref{thm-1.5}. Let $B=B_{\mathcal{R}}$ define Robin boundary
conditions. We clear the previous notation concerning the constants
$\varepsilon_\alpha^i$. Recall $\tilde B_{\mathcal{R}}\rho=\rho_1+\tilde S\rho_0$. Lemma \ref{lem-3.1} extends immediately to this setting,
after including the additional tensor
$S$ in the Weyl calculus, to yield:
\begin{lemma}\label{lem-4.1}
There exist universal constants $\varepsilon_\alpha^i$ and $d_\alpha^j$  so
that:\begin{enumerate}
\item $\int_{\partial M}\beta_{0,\alpha}^{\partial M}(\phi,\rho,D,{B_{\mathcal{R}}})dy=\int_{\partial
M}\varepsilon_\alpha^0\langle\phi_0,\rho_0\rangle dy$.
\smallbreak\item $\int_{\partial M}\beta_{1,\alpha}^{\partial M}(\phi,\rho,D,{B_{\mathcal{R}}})dy=\int_{\partial M}\{
\varepsilon_\alpha^1\langle\phi_1,\rho_0\rangle+\varepsilon_\alpha^2
\langle L_{aa}\phi_0,\rho_0\rangle+\varepsilon^3_\alpha\langle\phi_0,\rho_1\rangle$
\smallbreak $+d_\alpha^1\langle\phi_0,\tilde B_{\mathcal{R}}\rho\rangle\}dy$.
\smallbreak\item $
\int_{\partial M}\beta_{2,\alpha}^{\partial M}(\phi,\rho,D,{B_{\mathcal{R}}})dy
    =\int_{\partial M}\{\varepsilon_\alpha^4\langle\phi_2,\rho_0\rangle+\varepsilon_\alpha^5\langle L_{aa}\phi_1,\rho_0\rangle
    +\varepsilon_\alpha^6\langle E\phi_0,\rho_0\rangle$
\smallbreak$+\varepsilon_\alpha^7\langle\phi_0,\rho_2\rangle
   +\varepsilon_\alpha^8\langle L_{aa}\phi_0,\rho_1\rangle
    +\varepsilon_\alpha^9\langle\operatorname{Ric}_{mm}\phi_0,\rho_0\rangle
+\varepsilon_\alpha^{10}\langle L_{aa}L_{bb}\phi_0,\rho_0\rangle$\smallbreak
$+\varepsilon_\alpha^{11}\langle
L_{ab}L_{ab}\phi_0,\rho_0\rangle+\varepsilon_\alpha^{12}\langle \phi_{0;a},\rho_{0;a}\rangle
   +\varepsilon_\alpha^{13}\langle\tau\phi_0,\rho_0\rangle+\varepsilon_\alpha^{14}\langle\phi_1,\rho_1\rangle$\smallbreak
$+\langle d_\alpha^2\phi_1+d_\alpha^3 S\phi_0+d_\alpha^4 L_{aa}\phi_0,\tilde B_{\mathcal{R}}\rho\rangle\}dy$.
\end{enumerate}\end{lemma}

We begin our analysis by showing that all the constants $\varepsilon^i_\alpha$ vanish:

\begin{lemma}\label{lem-4.2}
\ \begin{enumerate}
\item If ${\tilde B}_{\mathcal{R}}\rho=0$, then $\partial_t\beta(\phi,\rho,D,B_{\mathcal{R}})(t)=-\beta(\phi,\tilde
D\rho,D,B_{\mathcal{R}})(t)$.
\item $\varepsilon_\alpha^0=\varepsilon_\alpha^1=\varepsilon_\alpha^2=0$.
\item $\varepsilon_\alpha^4=\varepsilon_\alpha^5=\varepsilon_\alpha^6
=\varepsilon_\alpha^7=\varepsilon_\alpha^9=\varepsilon_\alpha^{10}=
\varepsilon_\alpha^{11}=\varepsilon_\alpha^{12}=\varepsilon_\alpha^{13}=0$.
\item $\varepsilon_\alpha^3=\varepsilon_\alpha^8=\varepsilon_\alpha^{14}=0$.
\end{enumerate}
\end{lemma}

\begin{proof} Assertion (1) follows using the same arguments used to prove Lemma \ref{lem-3.6} (1); Equation (\ref{eqn-3.h}) then
generalizes to become
$$
{\textstyle\frac{1+k-\alpha}2}\int_{\partial M}\beta_{k,\alpha}^{\partial M}(\phi,\rho,D,B_{\mathcal{R}})dy
=-\int_{\partial M}\beta_{k-2,\alpha}^{\partial M}(\phi,\tilde D\rho,D,B_{\mathcal{R}})dy\,.
$$
We take $S=0$ and $\rho_1=0$; $\rho_0$ and $\rho_2$ are then arbitrary. Since $\beta_{-2,\alpha}^{\partial M}=0$
and $\beta_{-1,\alpha}^{\partial M}=0$, Assertion (2) follows. This implies that $\beta_{0,\alpha}^{\partial
M}(\phi,\rho,D,B_{\mathcal{R}})=0$ and a similar argument now establishes Assertion (3). We now take $S=-1$ and $\rho_0=\rho_1=1$ to
establish Assertion (4).
\end{proof}

The constants $d_\alpha^1$, $d_\alpha^2$, and $d_\alpha^3$ can be determined by a $1$-dimensional calculation. We adopt the following
notational conventions.  Let $M:=[0,1]$, let $A:=\partial_x+b$ where $b\in C^\infty(M)$ is real valued, let $A^*:=-\partial_x+b$,
let $D_1:=A^*A$, let $D_2:=AA^*$, and let $B_{\mathcal{R}}\phi:=A\phi|_{\partial M}$. The inward unit normal is
$\partial_x$ near $x=0$ and $-\partial_x$ near $x=1$. Thus this is a Robin boundary condition with $S(0)=b(0)$ and
$S(1)=-b(1)$.

\begin{lemma}\label{lem-4.3}
Let $\alpha\in\mathbb{C}-\mathbb{Z}$ with $\operatorname{Re}(\alpha)<0$. Adopt the notation established above.
\begin{enumerate}
\smallbreak\item $\partial_t\beta(\phi,\rho,D_1,B_{\mathcal{R}})(t)=-\beta(A\phi,A\rho,D_2,{B_{\mathcal{D}}})(t)$.
\item $\int_{\partial M}\beta_{k,\alpha}^{\partial M}(\phi,\rho,D_1,B_{\mathcal{R}})dy=-\frac{2}{1+k-\alpha}\int_{\partial M}
\beta_{k-1,\alpha+1}^{\partial M}(A\phi,A\rho,D_2,B_{\mathcal{D}})dy$.
\smallbreak\item $d_\alpha^1=\frac{2\alpha}{2-\alpha}c_{\alpha+1}$.
\item $d_\alpha^2\phi_1+d_\alpha^3S\phi_0=-\frac{2}{3-\alpha}c_\alpha\{(1-\alpha)\phi_1+S\phi_0\}$.
\end{enumerate}
\end{lemma}

\begin{proof} {We generalize the proof of Lemma 2.1.15 \cite{G04} where a similar result is established for $\alpha=0$. One has
that
$Ae^{-tD_{1,B_{\mathcal{R}}}}= e^{-tD_{2,B_{\mathcal{D}}}}A$ on sufficiently smooth functions. Thus we may establish Assertion (1) by
noting:
\begin{eqnarray*}
\partial_t\langle e^{-tD_{1,B_{\mathcal{R}}}}\phi,\rho\rangle &=& -\langle A^*Ae^{-tD_{1,B_{\mathcal{R}}}}\phi,\rho\rangle=
-\langle A^*e^{-tD_{2,B_{\mathcal{D}}}}A\phi,\rho\rangle\\ & = &-
\langle e^{-tD_{2,B_{\mathcal{D}}}}A\phi,A\rho\rangle
\end{eqnarray*}
where the middle equality is justified by the boundary
condition $B_{\mathcal{D}}$. }

Since $r^{\alpha}\phi\in C^\infty(V)$, we have $r^{\alpha+1}A\phi\in C^\infty(V)$. Thus
\begin{eqnarray*}
&&\sum_{k=0}^\infty{\textstyle\frac{1+k-\alpha}2}t^{(1+k-\alpha)/2-1}\int_{\partial M}\beta_{k,\alpha}^{\partial
M}(\phi,\rho,D,{B_{\mathcal{R}}})dy\\ &\sim&-\sum_{\ell=0}^\infty t^{(1+\ell-(\alpha+1))/2}\int_{\partial M}\beta_{\ell,\alpha+1}^{\partial
M}(A\phi,A\rho,D,{B_{\mathcal{R}}})dy\,.
\end{eqnarray*}
Setting $\ell=k-1$ and equating coefficients of $t^{(-1-k-\alpha)/2}$ yields Assertion (2).

The operators
$$D_1=-(\partial_x^2+(b^\prime-b^2))\quad\text{and}\quad D_2=-(\partial_x^2+(-b^\prime-b^2))$$
determine flat connections. We suppose that $b$, $\phi$, and $\rho$ vanish identically near $r=1$ so only the point $r=0$ is relevant.
We set $\phi_{-1}:=0$ and expand:
$$
\displaystyle\phi\sim\sum_{i=0}^\infty\phi_ir^{i-\alpha},\quad
\displaystyle A\phi\sim\sum_{i=0}^\infty\left\{(i-\alpha)\phi_i+b\phi_{i-1}\right\}r^{i-\alpha-1}\,.
$$
It now follows that $(A\phi)_0=-\alpha\phi_0$ and $(A\phi)_1=(1-\alpha)\phi_1+b\phi_0$. We apply
Assertion (2) with $k=1$ and $k=2$ and we apply Theorem \ref{thm-1.5} (1) to see:
\begin{eqnarray*}
&&\int_{\partial M}\langle d_\alpha^1\phi_0,A\rho\rangle dy=-\frac{2}{2-\alpha}c_{\alpha+1}
\int_{\partial M}\langle-\alpha\phi_0,A\rho\rangle dy,\\
&&\int_{\partial M}\langle d_\alpha^2\phi_1+d_\alpha^3b\phi_0,A\rho\rangle dy
=-\frac{2}{3-\alpha}c_\alpha\int_{\partial M}\langle(1-\alpha)\phi_1+b\phi_0,A\rho\rangle dy\,.
\end{eqnarray*}
Assertions (2) and (3) now follow.
\end{proof}

We extend Lemma \ref{lem-3.5} to the setting at hand to complete the proof of Theorem \ref{thm-1.5} (3).
\begin{lemma}\label{lem-4.4}
Adopt the notation established in Lemma
\ref{lem-3.5}. Let $S=\frac12\sum_af_a^\prime$ define Robin boundary conditions.
Then
\begin{enumerate}
\item $\int_{\partial M}\beta_{2,\alpha}^{\partial M}(\phi,\rho_M,D_M,B_{\mathcal{R}})dy=0$.
\smallbreak\item $d_\alpha^4=-\frac\alpha{3-\alpha}c_\alpha$.
\end{enumerate}
\end{lemma}

\begin{proof} Taking into account the change in the connection, we have that ${B_{\mathcal{R}}}$ on $M$ agrees with
the pure Neumann operator $B_{\mathcal{N}}$ defined by $S=0$ on $[0,1]$. {Since one has that $\rho dx=r\chi(r) dydr$,}
$$\beta(\phi,\rho,D_M,{B_{\mathcal{R}}})(t)=(2\pi)^{m-1}\beta(\phi,r\chi,-\partial_r^2,B_{\mathcal{N}})(t)\,.$$ Assertion (1) now
follows {as} $\int_{\partial[0,1]}\beta_{2,\alpha}^{\partial M}(\phi,r\chi,-\partial_r^2,B_{\mathcal{N}})dr=0$.

To prove Assertion (2), it is simply a matter of disentangling everything. We use the equations of structure derived in the proof of
Lemma \ref{lem-3.4} to see:
$$\begin{array}{llll}
\phi_0=1,&\phi_1=-\frac12\sum_af_a^\prime,&\rho_0=0,&\rho_1=1,\\
S=\textstyle\frac12\sum_af_a^\prime,&L_{aa}=-\sum_af_a^\prime.\\
\end{array}$$
We may now compute:
\begin{eqnarray*}
0=\int_{\partial
M}\{-\textstyle\frac{2}{3-\alpha}c_\alpha\{(1-\alpha)(-\frac12\sum_af_a^\prime)+\frac12\sum_af_a^\prime)
+d_\alpha^4(-\sum_af_a^\prime)\}dy.
\end{eqnarray*}
It now follows that $d_\alpha^4=-\frac\alpha{3-\alpha}c_\alpha$.
\end{proof}

\section*{Acknowledgements} The authors acknowledge support by the Isaac Newton Institution in
Cambridge during the Spectral Theory Program in July 2006 where this research began.
Research of P. Gilkey was also partially
supported by the Max Planck Institute for Mathematics in the Sciences (Germany) and by Project MTM2006-01432 (Spain).

\end{document}